\documentclass[final]{siamltex}
\usepackage{amsfonts}
\usepackage{graphics}
\usepackage{caption2}
\usepackage{amssymb}
\usepackage{psfrag}
\usepackage{amsmath,bm}
\usepackage{graphicx}
\usepackage{multirow}
\usepackage{cite}
\usepackage{subfigure}
\usepackage[normalem]{ulem}
\usepackage{color}
\usepackage{mathrsfs}
\usepackage{subeqnarray}
\usepackage{cases}

\allowdisplaybreaks[4]

\def\x{{\bm   x}}

\def\u{{\bm   u}}
\def\g{{\bm   g}}
\def\v{{\bm   v}}

\def\q{{\bm   q}}

\def\F{{\bm   F}}
\def\I{{\bm   I}}
\def\J{{\bm   J}}
\def\n{{\bm   n}}

\def\({\left(}
\def\[{\left[}
\def\){\right)}
\def\]{\right]}
\def\div{\nabla\cdot }
\def\grad{\nabla }

\newtheorem{lem}{Lemma}
\newtheorem{thm}{Theorem}

\numberwithin{equation}{section}
\numberwithin{thm}{section}
\numberwithin{lem}{section}

\begin{document}
\title{Entropy stable  modeling of  non-isothermal   multi-component diffuse-interface two-phase  flows with realistic equations of state\thanks{This work is  supported by    KAUST research fund to the
Computational Transport Phenomena Laboratory at KAUST.}
}

\author{Jisheng Kou\thanks{School of Mathematics and Statistics, Hubei Engineering  University, Xiaogan 432000, Hubei, China. } \and Shuyu Sun\thanks{Corresponding author. Computational Transport Phenomena Laboratory, Division of Physical Science and Engineering,
King Abdullah University of Science and
Technology, Thuwal 23955-6900, Kingdom of Saudi Arabia.   Email: {\tt shuyu.sun@kaust.edu.sa}.}
}

 \maketitle

\begin{abstract}
  In this paper, we consider mathematical modeling and numerical simulation of non-isothermal  compressible   multi-component diffuse-interface  two-phase  flows with realistic equations of state.  A general  model with general reference velocity  is derived  rigorously  through   thermodynamical laws and Onsager's reciprocal principle, and it is capable of  characterizing compressibility and partial miscibility between multiple  fluids.  We prove a  novel relation among the  pressure, temperature and chemical potentials, which results in  a new formulation of the  momentum conservation equation indicating   that the gradients of  chemical potentials and temperature become the primary driving force of the fluid motion except for the external forces.  
A key challenge in numerical simulation is to develop entropy stable numerical schemes preserving the   laws of thermodynamics. Based on  the convex-concave  splitting of Helmholtz free energy density with respect to molar densities and temperature,  we  propose an entropy stable numerical method, which solves the total energy balance equation directly, and thus, naturally satisfies the first law of thermodynamics.  Unconditional  entropy stability (the second law of thermodynamics) of the proposed method is proved by estimating the variations of Helmholtz free energy and kinetic energy with time steps.  Numerical results validate  the proposed  method.

\end{abstract}

\section{Introduction}

Various non-isothermal multi-component  two-phase   flows are ubiquitous  in nature and industry,  and thus their research carries broad and far-reaching significance.    An industrial example is the phase transition of hydrocarbon mixtures   in the reservoir; at specified thermodynamical conditions,  a hydrocarbon mixture  may split into gas and liquid (oil) to stay in an equilibrium state; when  the thermal enhanced oil recovery   is employed, intentionally introduced heat disrupts the equilibrium states and vaporizes part of the oil,  thereby changing the physical properties   such that oil  flows more freely through the reservoir \cite{chen2006multiphase,kou2017nonisothermal}.    Another example is the utilization of supercritical fluids   as solvents in chemical analysis and synthesis \cite{liu2016binaryPRE}. In  the natural world, many common  phenomena, such as boiling, evaporation and  condensation, are also related to physical properties and motion of  non-isothermal   two-phase   flows   \cite{liu2015liquid}.

 For  realistic fluids, the interfaces between multiple fluids always exist and play a very important role in the mass and energy transfer between different phases. 
 Partial miscibility of multiple fluids,  a common phenomenon displayed by realistic fluids in  experiments and practical processes, takes place through the interfaces. 
 Moreover, capillarity effect, a significant  mechanism  of flows in porous media,   is  also caused by the anisotropic attractive force of molecules on the interfaces  \cite{kousun2015CMA}. 
To describe the gas-liquid interfaces,  the diffuse-interface models for  multiphase    flows have been  developed in the literature. The pioneering work is that  the density-gradient contribution on the interfaces is introduced by van der Waals in the  energy density (see \cite{Onuki2007PRE} and the references therein), and on the base of it, the   Korteweg stress formulation is induced by composition gradients (see \cite{liu2015liquid,Onuki2007PRE} and the references therein).    Various phase-field models for immiscible and incompressible two-phase flows have been developed  and simulated  in the literature,  \cite{CahnHilliard1958,Emmerich2011PhaseField,Abels2012TwoPhaseModel,Bao2012FEM,Kim2009PhaseField,Alpak2016PhaseField,Guo2015phasefield,Ding2007PhaseField,Rasheed2013phasefield,Li2016MCH,Dong2017PhaseField,Zhang2016Phasefield,Yang2017MCH,Kim2012PhaseField,Boyer2006MCH,shen2016JCP} for instance. 

 Modeling and simulation of compressible multi-component  two-phase flows with partial miscibility and realistic equations of state (e.g. Peng-Robinson equation of state \cite{Peng1976EOS})  are intensively  studied  in recent years  \cite{qiaosun2014,kousun2015SISC,kousun2015CHE,fan2017componentwise,kouandsun2017modeling,Li2017IEQPREOS,Peng2017convexsplitting}.  The multi-component    models  with realistic equations of state are traditionally  applied   for simulation of many  problems in chemical and petroleum engineering, for example,  the phase equilibria calculations \cite{jindrova2013fast,mikyvska2015General,mikyvska2011new,mikyvska2012investigation,kousun2016Flash,kou2018nvtpc,Nagarajan1991,firoozabadi1999thermodynamics} and prediction of surface tension  \cite{miqueu2004modelling,kousun2015CMA,li2001prediction,firoozabadi1999thermodynamics}, but the fluid motion is never considered in these applications.  The models of compositional fluid flows in porous media, for example, \cite{polivka2014compositional,Moortgat2013compositional,Hoteit2013diffusion,Firoozabadi2007Diffusion}, describe the  fluid motion through Darcy's law, but  using the sharp interface.   A general diffuse interface model for compressible multi-component  two-phase flows with partial miscibility are developed in \cite{kouandsun2016multiscale,kouandsun2017modeling} based on   the thermodynamic laws and  realistic equations of state.  
 It uses molar densities as the primal state variables, and takes a general  thermodynamic pressure as  a function of the molar density and temperature, thereby  never suffering  from the difficulty of constructing the pressure  equation. However,  the temperature field  in this model is assumed to be homogeneous and constant.

There exist many situations, such as boiling, evaporation,   condensation and thermal enhanced oil recovery,  in which phase transitions and fluid motions are highly influenced by an inhomogeneous and variable  temperature field.  In \cite{Anderson1998DiffuseInterfaceMethods}, the non-isothermal diffuse-interface models for the single-component and binary fluids are developed   by including gradient contributions in the internal energy. In \cite{Onuki2005PRL,Onuki2007PRE}, Onuki  generalizes the van der Waals theory for the single-component fluids by including gradient contributions in both the internal energy and the entropy.
 Subsequently, improvements and applications of the model in \cite{Onuki2005PRL,Onuki2007PRE} are investigated in  \cite{liu2015liquid,Pecenko2011VanderWaalsModel,Bueno2016liquid,Qian2016heatflow,Chaudhri2014VDWApp,liu2016binaryPRE,Pecenko2010approachVanderWaals} and the other literature; especially, in  \cite{liu2015liquid}  a  continuum mechanics modeling framework for liquid-vapor  flows is rigorously derived using the thermodynamical laws. The non-isothermal diffuse-interface models are extended to  the compressible binary fluids  in \cite{liu2016binaryPRE,Gonnella2008binaryVDW}.   The aforementioned  research works are done    on the basis of the van der Waals equation of state, but rarely concerning the other equations of state, for example, the Peng-Robinson equation of state \cite{Peng1976EOS} extensively employed in petroleum and chemical industries due to its accuracy and consistency   for numerous realistic gas-liquid fluids including N$_2$, CO$_2$,  and hydrocarbons.  It is noted that different from the single-component fluids, a reference velocity,  such as mass-average velocity, molar-average velocity and so on,  usually needs to be selected for multi-component fluids \cite{Groot2015NET}, so the models  are expected to be compatible with general reference velocity.  However, up to now,   the rigorous  generalization of   the aforementioned models to  multi-component flows with general reference velocity is still an open problem.  
 
 In this paper, we will generalize the aforementioned model; more precisely,  we will derive a general non-isothermal multi-component diffuse-interface two-phase model based on  the thermodynamical laws and realistic equations of state (e.g. Peng-Robinson equation of state) with mathematical rigors.  A significant  feature of the general model is that it has a set of  unified formulations for general reference velocities and related mass diffusion fluxes. Moreover, a general  thermodynamic pressure, which is a function of the molar density and temperature, is used and consequently, it is free of constructing the pressure  equation.
 
 The entropy balance equation plays a fundamental role in the derivations of diffuse interface two-phase flow models \cite{Anderson1998DiffuseInterfaceMethods,Onuki2005PRL,Onuki2007PRE,liu2016binaryPRE,liu2015liquid}, by which we can apply the entropy production principle  (the second law of thermodynamics) to derive the forms of thermodynamical fluxes, including the stress tensor,  the mass diffusion and heat transfer fluxes. Different from the existing  derivation approach  using the Gibbs relation and other thermodynamical relations, we  derive the entropy balance equation directly from the total energy balance equation  (the first law of thermodynamics) based on general mass balance equations involving a general reference velocity.  The transport equation of Helmholtz free energy density is derived to further reduce the entropy balance equation into a form composed of conservative terms and entropy production terms, from which we derive the forms of thermodynamical fluxes using the  non-negativity principle of entropy production  and Onsager's reciprocal principle.
 Consequently, the derived general model  satisfies the first and second laws of thermodynamics and Onsager's reciprocal principle, thereby ensuring the thermodynamical  consistency.

In the  momentum conservation equations of the existing models, the pressure and surface tension are formulated as  the primary driving force of the fluid motion except for the external forces.  In this paper, we prove a  novel relation among the  pressure, temperature and chemical potentials, which leads to a new formulation of the  momentum conservation equation indicating   that the gradients of  chemical potentials and temperature become the primary driving force  except for the external forces.  It will be shown  that on the basis of the new formulation of momentum conservation equation, we can conveniently design   efficient, easy-to-implement and entropy stable numerical algorithms.  

A key and challenging   issue in numerical simulations is how to design an algorithm preserving  the  first and  second laws of thermodynamics obeyed by the model. 
To our best knowledge, there are only a few works regarding such algorithms in the literature. In  \cite{liu2015liquid},  a provably entropy-stable  numerical scheme has been developed   fundamentally on the concept of functional entropy variables. 
Recently, in \cite{kou2017nonisothermal}, using   the convex-concave splitting of  Helmholtz free energy density,  the authors propose an entropy stable numerical method for the single-component  fluids with the Peng-Robinson equation of state, and rigorously prove that the  first and  second laws of thermodynamics are preserved by this method.  But only single-component flows are considered in these existing methods.  In this paper, we focus on    the numerical  schemes preserving  the  laws of thermodynamics for the general multi-component flow model. 

The proposed numerical scheme will be based on the convex splitting approach, which is first proposed in \cite{Elliott1993ConvexSpliting,Eyre1998ConvexSplitting} and  has been  popularly employed in various phase-field models \cite{shen2015SIAM,Wise2009Convex,Eyre1998ConvexSplitting,Baskaran2013convexsplitting,Guo2016convexsplitting}. For isothermal single-component and multi-component    diffuse-interface models with Peng-Robinson equation of state,  the convex splitting schemes have been intensively studied recently \cite{qiaosun2014,fan2017componentwise,kouandsun2017modeling,kou2018nvtpc,Peng2017convexsplitting,Li2017convex}, and very recently, a convex splitting scheme for the non-isothermal single-component fluids  is also developed in \cite{kou2017nonisothermal}. But   the convex splitting approach is never studied yet for non-isothermal multi-component fluids with Peng-Robinson equation of state. In this paper, we will develop  the  convex-concave splitting of  Helmholtz free energy density based on Peng-Robinson equation of state; in particular,  we will prove that this  Helmholtz free energy density is always concave with respect to temperature. 

In the  algorithm proposed in \cite{kou2017nonisothermal},   the internal energy equation is solved instead of  the total energy balance equation. The proposed numerical algorithm in this paper will solve the total energy balance equation directly for ease of  preserving  the first law of thermodynamics.   
The key issue becomes  how to gain  entropy stability; in other words,     the method shall be designed to satisfy   the second law of thermodynamics.
It is well known that the convex-concave splitting schemes  usually result in the  energy-dissipation feature for   isothermal systems, whereas in this work we will  prove that the convex-concave splitting of Helmholtz free energy density with respect to molar densities  and temperature leads to the entropy stability of a numerical scheme for the  non-isothermal systems.  Another great challenge in designing the entropy-stable scheme is the very tightly  coupling relationship among  molar density, energy (temperature)  and velocity. Such relationship will be well treated through very careful  mathematical and physical observations.

The rest of this paper is organized as   follows. In Section 2, we will  introduce the mathematical  model for non-isothermal multi-component diffuse-interface two-phase flows, and rigorous  derivations  for this model will be  provided   in Section 3.  In Section 4, we will propose  an entropy stable numerical scheme based on the convex-concave splitting of Helmholtz free energy density and prove its entropy stability.  Numerical results will be provided  in Section 5 to validate the proposed method.   Finally,   some concluding remarks are given in Section 6.

\section{Non-isothermal multi-component two-phase flow model}\label{secModel}
In this section, we present the system of equations modeling  the non-isothermal multi-component two-phase flows with general reference velocity, which is composed of the mass balance equations, the momentum balance equation and total energy balance equation.    Moreover, a new  formulation of the momentum balance equation is derived  through  the relationship between the pressure, chemical potentials and temperature.

\subsection{Notations and thermodynamical relations}
We consider a fluid mixture  composed of $M$ components in a variable temperature field. The temperature is denoted by $T$. Let $n_i$ denote  the molar density of the $i$th component, and then we denote the molar density vector by $\n = [n_1,n_2,\cdots,n_M]^T$.  

The diffuse  interfaces, which always exist between two phases in a realistic fluid, play an extremely important  role in the heat and mass transfer between multiple phases. The key feature of diffuse interface models is to introduce the local density gradient  contribution in the energy density of inhomogeneous fluids. The  general  Helmholtz free energy density, denoted by $f$, is expressed as
\begin{subequations}\label{eqHelmholtzfreeenergydensity}
\begin{equation}\label{eqHelmholtzfreeenergydensity01}
 f(\n,T)=f_b(\n,T)+f_\grad(\n,T), 
\end{equation}
\begin{equation}\label{eqHelmholtzfreeenergydensity02}
 f_\grad(\n,T)=\frac{1}{2}\sum_{i,j=1}^Mc_{ij}(\n,T)\grad n_i\cdot\grad n_j, 
\end{equation}
\end{subequations}
where  $f_b$ stands for the Helmholtz free energy density of a bulk fluid  and $c_{ij}$ is the cross influence parameter generally relying on molar densities and temperature.   
 
The chemical potential of component $i$ is defined  as
\begin{eqnarray}\label{eqgeneralchemicalpotential}
    \mu_i=\(\frac{\delta f(\n,T)}{\delta n_i}\)_{T,n_1,\cdots,n_{i-1},n_{i+1},\cdots,n_M},~~i=1,\cdots, M,
\end{eqnarray}
where  $\frac{\delta}{\delta n_i}$ represents the variational derivative.
The  entropy density, denoted by  $s$, can be defined as 
\begin{eqnarray}\label{eqEntropyHelmholtz02}
    s = -\(\frac{\delta f}{\delta T}\)_\n.
\end{eqnarray}
We denote    the internal energy density by $\vartheta$.   The internal energy, entropy, and temperature have the following relation
\begin{eqnarray}\label{eqinternalenergydensity}
    \vartheta=f+sT.
\end{eqnarray}
The following relation between the pressure, Helmholtz free energy and chemical potential  \cite{kouandsun2017modeling} holds for the bulk and inhomogeneous fluids
\begin{eqnarray}\label{eqgeneralpressure}
    p=\sum_{i=1}^M\mu_in_i-f,
\end{eqnarray}
which allows us to define the general thermodynamical pressure. 

 According to \eqref{eqgeneralchemicalpotential}, the chemical potential of component $i$ can be deduced  from \eqref{eqHelmholtzfreeenergydensity} as
  \begin{equation}\label{eqgeneralchemicalpotentialForm}
  \mu_i=\mu_i^b-\sum_{j=1}^M\div\(c_{ij}\grad{n_j}\)+ \frac{1}{2}\sum_{j,k=1}^N\frac{\partial c_{jk}}{\partial n_i} \nabla n_j\cdot\nabla n_k,
  \end{equation}
  where $\mu_i^b=\(\frac{\partial  f_b(\n,T)}{\partial n_i}\)_{T,n_1,\cdots,n_{i-1},n_{i+1},\cdots,n_M}$. We get the formulation of the general pressure from \eqref{eqHelmholtzfreeenergydensity}, \eqref{eqgeneralpressure}  and \eqref{eqgeneralchemicalpotentialForm} as
\begin{eqnarray}\label{eqMultiComponentDefPresGeneralB}
    p 
&=&p_b-  \sum_{i,j=1}^Mn_i\div\(c_{ij}\grad{n_j}\)+ \frac{1}{2}\sum_{i,j,k=1}^Nn_i\frac{\partial c_{jk}}{\partial n_i} \nabla n_j\cdot\nabla n_k\nonumber\\
   &&-\frac{1}{2} \sum_{i,j=1}^Mc_{ij}\nabla n_i\cdot\nabla n_j,
\end{eqnarray}
where $p_b=\sum_{i=1}^Mn_i\mu_i^b-f_b.$

We denote by  $M_{w,i}$ the molar weight of component $i$, and  we further define the   mass density 
 of the mixture  as
\begin{eqnarray}\label{eqMultiCompononentTotalDensity}
  \rho=\sum_{i=1}^Mn_iM_{w,i}.
\end{eqnarray}
Let $g$ be the  absolute value of the gravity acceleration   and $h$ is  the height referred to a given reference platform. 

\subsection{Model equations}

 We now state the system of model equations,  which basically consists of  the mass balance equations, the momentum balance equation and total energy conservation  equation. First,  the mass balance equation   for component  $i$   is written as
\begin{equation}\label{eqGeneralNSEQMass}
\frac{\partial n_i}{\partial t}+\div\(\u n_i\)+\div \J_i=0,
\end{equation}
where $\u$ is a reference velocity and $\J_i$ is the diffusion flux of component $i$. 

Let us define the stress tensor 
\begin{equation}\label{eqTotalStressB}
 \bm\sigma=p\I-\bm\tau,~~ \bm\tau\(\u\)= \(\lambda\div\u\)\I+\eta\bm\varepsilon(\u),~~~\bm\varepsilon(\u)=\nabla\u+\nabla\u^T,
\end{equation} 
where $\I$ is the identity tensor,  $\xi$ is the volumetric  viscosity,   $\eta$ is the shear viscosity and $\lambda=\xi-\frac{2}{3}\eta$. We assume that $\eta>0$ and $\lambda>0$. 
The  momentum balance equation is
\begin{equation}\label{eqGeneralNSEQ}
\rho\(\frac{\partial  \u}{\partial t}+ \u\cdot\grad{ \u}\)+\sum_{i=1}^MM_{w,i}\J_i\cdot\grad\u=-\div\bm\sigma
-\sum_{i,j=1}^M\div\(c_{ij}\nabla n_i\otimes\nabla n_j\)+ \rho\g,
\end{equation}
where $\g=-g\grad h$. Thanks to the following relation (which will be proved in Sub-section \ref{secRelationBetweenChPtlTempPres})
\begin{eqnarray}\label{eqRelationBetweenChPtlTempPres}
\grad p+\sum_{i,j=1}^M\div\(c_{ij}\grad n_i\otimes\grad n_j\)= \sum_{i=1}^Mn_i\grad \mu_i+s\grad T 
 ,
\end{eqnarray} 
 we obtain a new formulation of  the momentum balance equation       as
\begin{equation}\label{eqMultiComponentMomentumConserve1C07}
 \rho\(\frac{\partial  \u}{\partial t}+ \u\cdot\grad{ \u}\)+\sum_{i=1}^MM_{w,i}\J_i\cdot\grad\u=-\sum_{i=1}^Mn_i\grad \mu_i-s\grad T+\div\bm{\tau}(\u)+\rho\g ,
\end{equation}
which indicates that the  fluid motion is driven by the gradients of chemical potentials and temperature.

We denote the total energy density $e_t=\vartheta+\frac{1}{2}\rho|\u|^2+\rho gh$.   The total energy conservation  equation is expressed as
 \begin{equation}\label{eqTotalEnergyConservation}
 \frac{\partial e_t}{\partial t} + \div\(\u e_t+\bm\sigma\cdot\u\)=-\div\(\q-\bm\pi\) ,
\end{equation}
where $\q$ is the  heat flux and
\begin{eqnarray}\label{eqGenHeatFlux}
\bm\pi =\sum_{i,j=1}^Mc_{ij}\frac{\partial n_i}{\partial t}\grad{n_j} -\frac{1}{2}\sum_{i=1}^MM_{w,i}|\u|^2\J_i .
\end{eqnarray}

The mass diffusion fluxes in \eqref{eqGeneralNSEQMass} and the  heat flux in \eqref{eqTotalEnergyConservation} are assumed to be linearly related to $\grad \frac{\mu_i}{T}$ and $\grad \frac{1}{T}$ as
\begin{subequations}\label{eqDiffusionHeatFluxes}
\begin{equation}\label{eqDiffusionHeatFluxes01}
\J_i=-\sum_{j=1}^M\mathcal{L}_{i,j}\grad \frac{\mu_j+M_{w,j}gh}{T} + \mathcal{L}_{i,M+1}\grad \frac{1}{T},
\end{equation}
\begin{equation}\label{eqDiffusionHeatFluxes02}
\q=-\sum_{j=1}^M\mathcal{L}_{M+1,j}\grad \frac{\mu_j+M_{w,j}gh}{T} + \mathcal{L}_{M+1,M+1}\grad \frac{1}{T}.
\end{equation}
\end{subequations}
The mobility matrix $\bm{\mathcal{L}}=\(\mathcal{L}_{i,j}\)_{i,j=1}^{M+1}$ shall be symmetric in terms of Onsager's reciprocal principle. Moreover, the second law of thermodynamics requires that it shall be   positive definite or  positive semi-definite. 

For the boundary conditions, we assume that   all boundary terms will vanish when integrating by parts is performed; for example, we can use homogeneous Neumann boundary conditions or periodic boundary conditions.

We have the following comments on  the above model.
\begin{enumerate}
\item The model has thermodynamically-consistent unified formulations for the general reference velocity and the mass diffusion and heat fluxes.
\item The model can characterize the compressibility,  partial miscibility and heat transfer between different phases through the diffuse interfaces.
\item It is different from the case of the pure substance that there exist many choices of  the reference velocity $\u$ and the diffusion flux $\J_i$  for a multi-component mixture. Moreover, $\u$ and $\J_i$ have the tightly dependent  relations as shown below.
\item For the special case  of $\J_i$  taken such that $\sum_{i=1}^MM_{w,i}\J_i=0$,  the terms $\sum_{i=1}^MM_{w,i}\J_i\cdot\grad\u$ in \eqref{eqGeneralNSEQ} and $\sum_{i=1}^MM_{w,i}|\u|^2\J_i$  in \eqref{eqGenHeatFlux} will vanish.  However, in general cases, these terms are essential to ensure the thermodynamical consistency. 
\item Due to the use of  the general  thermodynamic pressure, it is not necessary  to construct the pressure  equation, indeed,  the pressure  can be explicitly   calculated from \eqref{eqMultiComponentDefPresGeneralB} if the molar density and temperature are available.  
\end{enumerate}

We now present  a few choices of the mass diffusion fluxes and heat flux. In general, these flux formulations  shall be chosen such that  the mobility matrix $\bm{\mathcal{L}}$ is symmetric,  positive definite or  positive semi-definite.  We take $\mathcal{L}_{i,M+1}=0$ for $1\leq i\leq M$,  and furthermore, we take $\mathcal{L}_{M+1,M+1}=\mathcal{K}T^2$, where $\mathcal{K}$ is the  Fourier thermal conductivity coefficient of the mixture. As a result, we obtain 
 the heat flux and the diffusion flux  of component $i$  as  
\begin{equation}\label{eqMultiCompononentMassConserveDiffusion}
 \q=-\mathcal{K}\grad{T},~~~ \bm J_{i}=-\sum_{j=1}^M\mathcal{L}_{ij}\nabla \frac{\mu_j+M_{w,j}gh}{T} ,~~i=1,\cdots, M.
\end{equation}
 To determine the diffusion flux $\bm J_{i}$ precisely,  we give two typical choices of $\mathcal{L}_{ij}~(1\leq i,j\leq M)$, which correspond to molar-average velocity and mass-average velocity respectively. 
\begin{description}
\item[(J1)]  In the first case,    the diffusion mobility parameters are taken as 
\begin{equation}\label{eqMultiCompononentMassConserveDiffusionChoiceA}
\mathcal{L}_{ii} = \sum_{j=1}^M\frac{\mathcal{D}_{ij} n_in_j}{nR},~~~~~\mathcal{L}_{ij} = -\frac{\mathcal{D}_{ij} n_in_j}{nR},~~j\neq i, ~~1\leq i,j\leq M,
\end{equation}
where $R$ stands for the universal gas constant and   the mole diffusion coefficients $\mathcal{D}_{ij}$ satisfy $\mathcal{D}_{ii}=0$ and $\mathcal{D}_{ij}=\mathcal{D}_{ji}>0$ for $i\neq j$.  In this case, we have  $\sum_{i=1}^M\J_i=0$, which means that the reference velocity  is  the molar-average velocity.
\item[(J2)]  In the second case,  we take
\begin{equation}\label{eqMultiCompononentMassConserveDiffusionChoiceB}
\mathcal{L}_{ii} = \sum_{j=1}^M\frac{\mathscr{D}_{ij}n_i\rho_j}{M_{w,i}\rho R},~~~~~\mathcal{L}_{ij} = -\frac{\mathscr{D}_{ij}n_in_j}{\rho R},~~j\neq i, ~~1\leq i,j\leq M,
\end{equation}
where    $\mathscr{D}_{ij}$ are the mass diffusion coefficients satisfying  $\mathscr{D}_{ii}=0$ and $\mathscr{D}_{ij}=\mathscr{D}_{ji}>0$ for $i\neq j$.   In this case, it holds that  $\sum_{i=1}^M M_{w,i}\J_i=0$,  so the reference velocity becomes  the mass-average velocity.
\end{description}

 It is easy to prove that the above choices of $\bm{\mathcal{L}}$ is symmetric   positive semidefinite, and consequently,  they obey Onsager's reciprocal principle \cite{Groot2015NET} and the second law of thermodynamics.

\section{Derivations of the model}
In this section,  we show the rigorous  derivations of the model equations given in Section \ref{secModel}.  The component mass balance equations \eqref{eqGeneralNSEQMass} are assumed to hold with general reference velocity, but the mass diffusion fluxes are undetermined.  The formulations of the mass diffusion fluxes, the momentum balance equation and total energy conservation equation will be derived using the laws of thermodynamics and Onsager's reciprocal principle.

 \subsection{Primary thermodynamical equations}
In a  time-dependent volume $V(t)$,  we define the internal energy  and  kinetic energy  ($E$) within $V(t)$ as
\begin{eqnarray}\label{eqDefKineticEnergy}
  U=\int_{V(t)} \vartheta dV,~~~E=\frac{1}{2}\int_{V(t)}\rho|\u|^2dV .
  \end{eqnarray}
 The gravitational potential energy   has the form
\begin{eqnarray}\label{eqGravitationalPotentialEnergy}
   H =\int_{V(t)} \rho gh dV.
\end{eqnarray}
We recall 
the first law of thermodynamics   
\begin{eqnarray}\label{eqFirstThermodynamicsLawUE}
    \frac{d(U+E+H)}{dt}=\frac{d_\_W}{dt}+\frac{d_\_Q}{dt},
\end{eqnarray}
where    $W$ is the work done by the  face force $\F_t$, and $Q$ is the heat transfer from external environment of $V(t)$. The work done by $\F_{t}$     is expressed as
\begin{eqnarray*}\label{eqMultiCompononentFirstThermodynamicsLaw07}
  \frac{d_\_W}{dt}=   \int_{\partial V(t)} \F_{t}\cdot\u d\bm s .
\end{eqnarray*}
Cauchy's relation between face force $\F_{t}$ and the stress tensor $\bm\sigma$ of  component $i$  gives $\F_{t}=-\bm\sigma\cdot\bm\nu$, and as a result, 
\begin{equation}\label{eqReynoldsMultiCompWork}
  \frac{d_\_W}{dt}= - \int_{\partial V(t)} \(\bm\sigma\cdot\bm\nu\)\cdot\u d\bm s = - \int_{ V(t)} \div\(\bm\sigma\cdot\u\) dV ,
\end{equation}
where $\bm\nu$ is the unit normal vector towards  the outside  of  $V(t)$.  We note that  the other external forces   are ignored in this work, but the model derivations can be easily  extended to the cases in the presence of additional external forces. 
The   heat flux is   expressed 
as  
\begin{eqnarray}\label{eqHeatTransfer} 
  \frac{d_\_Q}{dt}=-\int_{\partial V(t)}\bm\phi_q\cdot\bm\nu d\bm s=-\int_{ V(t)}\div\bm\phi_q  dV,
\end{eqnarray}
where $\bm\phi_q$ represents the general heat flux.

Applying the Reynolds transport theorem and the Gauss divergence theorem,  we deduce that
\begin{eqnarray}\label{eqReynoldsMultiComp01}
  \frac{dU}{dt}&=&\int_{V(t)}\frac{\partial \vartheta}{\partial t}dV+\int_{V(t)}\div\(\u \vartheta\)dV\nonumber\\
&=&\int_{V(t)}\(\frac{\partial f}{\partial t}+\div\(\u f\)\)dV+\int_{V(t)}T\(\frac{\partial s}{\partial t}+\div\(\u s\)\)dV\nonumber\\
 &&+\int_{V(t)}s\(\frac{\partial T}{\partial t}+\u \cdot\grad T\)dV,
\end{eqnarray}
where we have also used the relation $\vartheta=f+Ts$.
We can also derive that
\begin{eqnarray}\label{eqReynoldsMultiComp03a}
  \frac{dE}{dt}
  &=&\frac{1}{2}\int_{V(t)}\frac{\partial \(\rho \u\cdot\u\)}{\partial t}dV+\frac{1}{2}\int_{V(t)}\div\(\u \(\rho \u\cdot\u\)\)dV\nonumber\\
   &=&\int_{V(t)}\(\rho \u\cdot\frac{\partial \u}{\partial t}+\frac{1}{2}\u\cdot\u\frac{\partial\rho}{\partial t}\)dV\nonumber\\
   &&+\frac{1}{2}\int_{V(t)}\big(\(\rho \u\cdot\u\)\div\u+\(\u\cdot\u\)\u\cdot\grad\rho+ \rho\u\cdot\grad\( \u\cdot\u\)\big)dV\nonumber\\
   &=&\int_{V(t)}\(\rho \u\cdot\frac{\partial \u}{\partial t}+\frac{1}{2}\u\cdot\u\frac{\partial\rho}{\partial t}\)dV\nonumber\\
   &&+\frac{1}{2}\int_{V(t)}\big(\(\u\cdot\u\)\div\(\rho\u\)+ 2\rho\u\cdot\(\u\cdot\grad\u\)\big)dV\nonumber\\
   &=&\int_{V(t)}\rho\u\cdot\(\frac{\partial \u}{\partial t}+\u\cdot\grad\u\)dV+\frac{1}{2}\int_{V(t)}\u\cdot\u\(\frac{\partial\rho}{\partial t}+\div\(\rho\u\)\)dV.
\end{eqnarray}
On the other hand,   the following  overall mass balance equation can be obtained from the component mass balance equations \eqref{eqGeneralNSEQMass} 
\begin{eqnarray}\label{eqMultiCompononentMassConserveMass}
\frac{\partial \rho}{\partial t}+\div(\rho\u)+\sum_{i=1}^MM_{w,i}\div\J_i=0.
\end{eqnarray} 
 Substituting   \eqref{eqMultiCompononentMassConserveMass}   into \eqref{eqReynoldsMultiComp03a} yields
\begin{eqnarray}\label{eqReynoldsMultiComp03}
  \frac{dE}{dt}
  &=&\int_{V(t)}\u\cdot\(\rho\frac{D \u}{D t}+\sum_{i=1}^MM_{w,i}\J_i\cdot\grad\u\)dV\nonumber\\
  &&-\frac{1}{2}\int_{V(t)}\sum_{i=1}^MM_{w,i}\div\( |\u|^2\J_i\) dV,
\end{eqnarray}
where  $\frac{D \u}{D t}=\frac{\partial  \u}{\partial t}+ \u\cdot\grad{ \u}.$

 Substituting \eqref{eqReynoldsMultiCompWork}, \eqref{eqHeatTransfer}, \eqref{eqReynoldsMultiComp01}, \eqref{eqReynoldsMultiComp03} into  \eqref{eqFirstThermodynamicsLawUE},  and taking into account   the arbitrariness of $V(t)$, we obtain
\begin{eqnarray}\label{eqinitialTotalenergyEqn01}
 &&T\(\frac{\partial s}{\partial t}+\div(\u s)\)+s\frac{D T}{D t}+\frac{\partial  (f+\rho gh)}{\partial t}+\div\(\u  (f+\rho gh)\)\nonumber\\
&&~~
 +\u\cdot\(\rho\frac{D \u}{D t}+\sum_{i=1}^MM_{w,i}\J_i\cdot\grad\u\)-\frac{1}{2}\sum_{i=1}^MM_{w,i}\div\( |\u|^2\J_i\)\nonumber\\
&&~~ =-\div\bm\phi_q - \div\(\bm\sigma\cdot\u\),
\end{eqnarray}
where $\frac{D T}{D t}$ is the material derivative as $$\frac{D T}{D t}=\frac{\partial T}{\partial t}+\u\cdot\grad{T}.$$
Furthermore,    we rewrite \eqref{eqinitialTotalenergyEqn01} as
\begin{eqnarray}\label{eqEntropyMultiGrad01}
 \frac{\partial s}{\partial t}+\div(\u s) &=&-\frac{1}{T}\div\bm\phi_q - \frac{1}{T}\bm\sigma^T:\grad\u-\frac{s}{T}\frac{D T}{D t}\nonumber\\
 &&
 -\frac{1}{T}\(\frac{\partial  (f+\rho gh)}{\partial t}+\div\(\u  (f+\rho gh)\)\)\nonumber\\
 &&+\frac{1}{2T}\sum_{i=1}^MM_{w,i}\div\( |\u|^2\J_i\)\nonumber\\
 &&-\frac{\u}{T}\cdot\(\rho\frac{D \u}{D t}+\sum_{i=1}^MM_{w,i}\J_i\cdot\grad\u+\div\bm\sigma\)
.
\end{eqnarray}
 
 \subsection{Entropy  equation}
In order to derive the model equations using the second law of thermodynamics, we will deduce an entropy equation, which consists of conservative  and   non-negative terms. For this purpose,  we  need to derive the transport equation of Helmholtz free energy density to reduce \eqref{eqEntropyMultiGrad01}. 
 
 We define  ${\gamma}=\gamma_b+\gamma_\grad$, where $\gamma_b=\(\frac{\partial f_b}{\partial T}\)_\n$  and $$\gamma_\grad = \(\frac{\delta f_\grad}{\delta T}\)_\n=\frac{1}{2} \sum_{i,j=1}^M\frac{\partial c_{ij}}{\partial T}\nabla n_i\cdot\nabla n_j.$$
 
Firstly, 
using   the component mass balance equations, we  derive the transport equation of  $f_b$   as
\begin{eqnarray}\label{eqMultiCompononentHelmholtzMultiTransport01}
 \frac{\partial f_b}{\partial t}+\div(f_b\u)
 &=&\sum_{i=1}^M\mu_i^b\frac{\partial n_i}{\partial t}+\gamma_b\frac{\partial T}{\partial t}+f_b\div\u+\u\cdot\grad f_b\nonumber\\
 &=&-\sum_{i=1}^M\mu_i^b\(\div(n_i\u)+\div\J_i\)+\gamma_b\frac{\partial T}{\partial t}\nonumber\\
 &&+f_b\div\u
 +\sum_{i=1}^M\u\cdot\mu_i^b\grad n_i+ \u\cdot\gamma_b\grad T\nonumber\\
 &=&-\sum_{i=1}^M\(\mu_i^bn_i\div \u +\u\cdot\mu_i^b\grad n_i+\mu_i^b\div\J_i\)\nonumber\\
 &&+\gamma_b\frac{D T}{D t}+f_b\div\u
 +\sum_{i=1}^M\u\cdot\mu_i^b\grad n_i \nonumber\\
 &=&-p_b\div\u-\sum_{i=1}^M\mu_i^b\div\J_i+\gamma_b\frac{D T}{D t}.
\end{eqnarray}
For the gradient contribution of Helmholtz free energy density, we  can derive 
\begin{eqnarray}\label{eqMultiCompononentHelmholtzMultiTransport02}
\frac{\partial f_\grad}{\partial t}
 &=&\frac{1}{2}\frac{\partial\(\sum_{i,j=1}^M c_{ij} \nabla n_i\cdot\nabla n_j\)}{\partial{t}}\nonumber\\
 &=&\frac{1}{2}\sum_{i,j=1}^M \frac{\partial c_{ij}}{\partial T}\frac{\partial T}{\partial t} \nabla n_i\cdot\nabla n_j+\frac{1}{2}\sum_{i,j,k=1}^N\frac{\partial c_{jk}}{\partial n_i} \frac{\partial n_i}{\partial t}\nabla n_j\cdot\nabla n_k\nonumber\\
 &&+\sum_{i,j=1}^M c_{ij} \nabla\frac{\partial n_i}{\partial{t}}\cdot\nabla  n_j\nonumber\\
 &=&\gamma_\grad\frac{\partial T}{\partial t}- \frac{1}{2}\sum_{i,j,k=1}^N\frac{\partial c_{jk}}{\partial n_i} \(\div(n_i\u)+\div\J_i\)\nabla n_j\cdot\nabla n_k\nonumber\\
 &&-\sum_{i,j=1}^M c_{ij} \nabla n_j\cdot\nabla\big(\div(n_i\u)+\div\J_i\big)\nonumber\\
&=&\gamma_\grad\frac{\partial T}{\partial t}-\sum_{i,j=1}^M\div\big(\(\div\(\u n_i\)\)c_{ij}\grad{n_j}\big)\nonumber\\
&&- \frac{1}{2}\sum_{i,j,k=1}^N\frac{\partial c_{jk}}{\partial n_i} \(n_i\div\u+\u\cdot\grad n_i+\div\J_i\)\nabla n_j\cdot\nabla n_k\nonumber\\
 &&+\sum_{i,j=1}^M\div\(\u n_i\)\div\(c_{ij}\grad{n_j}\)-\sum_{i,j=1}^M c_{ij} \nabla n_j\cdot\nabla\(\div\J_i\)\nonumber\\
&=&\gamma_\grad\frac{\partial T}{\partial t}-\sum_{i,j=1}^M\div\big(\(\div\(\u n_i\)\)c_{ij}\grad{n_j}\big)\nonumber\\
 &&- \frac{1}{2}\sum_{i,j,k=1}^N\frac{\partial c_{jk}}{\partial n_i} \(n_i\div\u+\u\cdot\grad n_i+\div\J_i\)\nabla n_j\cdot\nabla n_k\nonumber\\&&+\sum_{i,j=1}^Mn_i\(\div\u \)\div\(c_{ij}\grad{n_j}\)
+\sum_{i,j=1}^M\(\u\cdot\grad{n_i}\)\div\(c_{ij}\grad{n_j}\)\nonumber\\
 &&-\sum_{i,j=1}^M  \div\big(\(\div\J_i\)c_{ij}\grad{n_j}\big)+\sum_{i,j=1}^M  \(\div\J_i\)\div\(c_{ij}\grad{n_j}\),
 \end{eqnarray}
 and
 \begin{align}\label{eqMultiCompononentHelmholtzMultiTransport03}
\div(f_\grad\u)
 &=\frac{1}{2}\div\(\u\sum_{i,j=1}^M c_{ij} \nabla n_i\cdot\nabla n_j\)\nonumber\\
 &=\frac{1}{2}\(\sum_{i,j=1}^M c_{ij} \nabla n_i\cdot\nabla n_j\)\div\u+\frac{1}{2}\u\cdot\grad\(\sum_{i,j=1}^M c_{ij} \nabla n_i\cdot\nabla n_j\).
 \end{align}
 Combining  \eqref{eqMultiCompononentHelmholtzMultiTransport01}-\eqref{eqMultiCompononentHelmholtzMultiTransport03}, we deduce the transport equation of the   Helmholtz free energy $ f$  as
\begin{eqnarray}\label{eqMultiCompononentHelmholtzMultiTransport05}
 \frac{\partial  f}{\partial t}+\div( f\u)
 &=&\frac{\partial f_b}{\partial t}+\div(f_b\u)+\frac{\partial f_\grad}{\partial t}+\div(f_\grad\u)\nonumber\\
 &=&\gamma_b\frac{D T}{D t}-p_b\div\u-\sum_{i=1}^M\mu_i^b\div\J_i\nonumber\\
 &&+\gamma_\grad\frac{\partial T}{\partial t}-\sum_{i,j=1}^M\div\big(\(\div\(\u n_i\)\)c_{ij}\grad{n_j}\big)\nonumber\\
 &&- \frac{1}{2}\sum_{i,j,k=1}^N\frac{\partial c_{jk}}{\partial n_i} \(n_i\div\u+\u\cdot\grad n_i+\div\J_i\)\nabla n_j\cdot\nabla n_k\nonumber\\
 &&+\sum_{i,j=1}^M\(\div\u \)n_i\div\(c_{ij}\grad{n_j}\)
+\sum_{i,j=1}^M\(\u\cdot\grad{n_i}\)\div\(c_{ij}\grad{n_j}\)\nonumber\\
 &&-\sum_{i,j=1}^M  \div\big(\(\div\J_i\)c_{ij}\grad{n_j}\big)+\sum_{i,j=1}^M  \(\div\J_i\)\div\(c_{ij}\grad{n_j}\)\nonumber\\
 &&+\frac{1}{2}\(\sum_{i,j=1}^M c_{ij} \nabla n_i\cdot\nabla n_j\)\div\u+\frac{1}{2}\u\cdot\grad\(\sum_{i,j=1}^M c_{ij} \nabla n_i\cdot\nabla n_j\)\nonumber\\
&=&- p ~\div\u+\gamma_b\frac{D T}{D t}-\sum_{i=1}^M\mu_i\div\J_i+\gamma_\grad\frac{\partial T}{\partial t}\nonumber\\
 &&+\sum_{i,j=1}^M\div\(\frac{\partial n_i}{\partial t}c_{ij}\grad{n_j}\)
 - \frac{1}{2}\sum_{i,j,k=1}^N \(\u\cdot\grad n_i\) \frac{\partial c_{jk}}{\partial n_i}\nabla n_j\cdot\nabla n_k\nonumber\\
 &&+\sum_{i,j=1}^M\(\u\cdot\grad{n_i}\)\div\(c_{ij}\grad{n_j}\)+\frac{1}{2}\u\cdot\grad\(\sum_{i,j=1}^M c_{ij} \nabla n_i\cdot\nabla n_j\),\nonumber\\
\end{eqnarray}
where we have also used the formulation  of the general pressure $p$.
Using the  identity \eqref{eqidentity}, we deduce that
\begin{eqnarray}\label{eqMultiCompononentHelmholtzMultiTransport04}
&&\sum_{i,j=1}^M\(\grad{n_i}\)\div\(c_{ij}\grad{n_j}\)+\frac{1}{2}\grad\(\sum_{i,j=1}^M c_{ij} \nabla n_i\cdot\nabla n_j\)\nonumber\\
 &&~~=\sum_{i,j=1}^M\(\grad{n_i}\)\div\(c_{ij}\grad{n_j}\)+\frac{1}{2}\sum_{i,j=1}^M\(\nabla n_i\cdot\nabla n_j\) \grad{c}_{ij} \nonumber\\
 &&~~~~~+\frac{1}{2}\sum_{i,j=1}^M c_{ij} \grad\(\nabla n_i\cdot\nabla n_j\)\nonumber\\
 &&~~=\gamma_\grad\grad{T}  +\frac{1}{2}\sum_{i,j,k=1}^N \(\grad n_i \)\frac{\partial c_{jk}}{\partial n_i}\nabla n_j\cdot\nabla n_k+\sum_{i,j=1}^M\div c_{ij}\(\nabla n_i\otimes\nabla n_j\).
  \end{eqnarray}
 Substituting \eqref{eqMultiCompononentHelmholtzMultiTransport04} into  \eqref{eqMultiCompononentHelmholtzMultiTransport05}, we obtain
  \begin{eqnarray}\label{eqFinalMultiCompononentHelmholtzMultiTransport}
 \frac{\partial  f}{\partial t}+\div( f\u)
 &=&- p ~\div\u+ \gamma\frac{D T}{D t}-\sum_{i=1}^M\mu_i\div\J_i +\sum_{i,j=1}^M\div\(\frac{\partial n_i}{\partial t}c_{ij}\grad{n_j}\)\nonumber\\
 && +\u\cdot\(\sum_{i,j=1}^M\div c_{ij}\(\nabla n_i\otimes\nabla n_j\)\)
 .
\end{eqnarray}

The transport equation of the gravity potential energy can be easily  derived as
\begin{eqnarray}\label{eqMultiCompononentGravity01}
 \frac{\partial \rho gh}{\partial t}+\div(\rho gh\u)
 &=&gh\(\frac{\partial \rho}{\partial t}+\div(\rho\u)\)+\u\rho g\grad h\nonumber\\
 &=&-\sum_{i=1}^MM_{w,i}gh\div\J_i-\rho \g\cdot\u.
\end{eqnarray}
Substituting    \eqref{eqFinalMultiCompononentHelmholtzMultiTransport} and \eqref{eqMultiCompononentGravity01}  into \eqref{eqEntropyMultiGrad01},  we  reduce the entropy balance equation  as
\begin{align}\label{eqEntropyMultiGrad02}
 &\frac{\partial s}{\partial t}+\div(\u s) =-\div\(\frac{\bm\phi_q}{T}\)+\bm\phi_q\cdot \grad\frac{1}{T}  - \frac{1}{T}\(\bm\sigma-p\I\):\grad\u-\frac{s+\gamma}{T}\frac{D T}{D t}\nonumber\\
&~~
   +\frac{1}{2}\sum_{i=1}^MM_{w,i}\div\( \frac{1}{T}|\u|^2\J_i\)-\frac{1}{2}\sum_{i=1}^MM_{w,i}|\u|^2\J_i\cdot\grad \frac{1}{T}\nonumber\\
 &~~ +\sum_{i=1}^M\div\(\frac{\mu_i+M_{w,i}gh}{T} \J_i\)-\sum_{i,j=1}^M\div\(\frac{1}{T}\frac{\partial n_i}{\partial t}c_{ij}\grad{n_j}\)\nonumber\\
 &~~+\sum_{i,j=1}^M\(c_{ij}\frac{\partial n_i}{\partial t}\grad{n_j}\)\cdot \grad\frac{1}{T}  -\sum_{i=1}^M\J_i\cdot\grad \frac{\mu_i+M_{w,i}gh}{T}\nonumber\\
 &~~-\frac{\u}{T}\cdot\(\rho\frac{D \u}{D t}+\sum_{i=1}^MM_{w,i}\J_i\cdot\grad\u+\div\bm\sigma+\sum_{i,j=1}^M\div c_{ij}\(\nabla n_i\otimes\nabla n_j\)-\rho\g\).
\end{align}

\subsection{Physical principles of the model derivations}

The second law of thermodynamics  states that  the  total entropy can never decrease over time for an isolated system, so the non-conservative terms in  \eqref{eqEntropyMultiGrad02} shall be non-negative. The fourth term on the right-hand side of \eqref{eqEntropyMultiGrad02} will vanish due to the relation $s=-\gamma$ given in \eqref{eqEntropyHelmholtz02}. Thanks to Galilean invariance, we must formulate the  momentum conservation equation  as the form of the equation \eqref{eqGeneralNSEQ} with an undetermined stress tensor,
thereby making  the last non-conservative term on the right-hand side of \eqref{eqEntropyMultiGrad02} disappear. In order  to find the detailed formulation of the stress tensor, we consider the third term  on the right-hand side of \eqref{eqEntropyMultiGrad02}, which is a non-conservative term and thus shall be non-negative. This term shall vanish for the reversible process, and in this case,  the total stress   only contains the reversible part, i.e., the pressure  $p\I$. For the realistic viscous fluids,   in terms of Newtonian fluid theory, the Cauchy stress tensor $\bm\tau\(\u\)$ given in \eqref{eqTotalStressB} shall be included in  the total stress, and consequently, we get $\bm\sigma=p\I-\bm\tau\(\u\)$, which ensures that the third term  on the right-hand side of \eqref{eqEntropyMultiGrad02} is  always non-negative
\begin{eqnarray}\label{eqTotalStressCondition}
   - \frac{1}{T}\(\bm\sigma-p\I\):\grad\u=\frac{1}{T}\bm\tau\(\u\):\grad\u=\frac{\lambda}{T}|\div\u|^2+\frac{\eta}{2T}|\bm\varepsilon(\u)|^2\geq0.
\end{eqnarray}
So far the momentum balance equation reaches its  complete form \eqref{eqGeneralNSEQ} with the stress tensor \eqref{eqTotalStressB}.

 W define the  heat flux
\begin{eqnarray}\label{eqGenHeatFlux}
\q&=&\bm\phi_q-\frac{1}{2}\sum_{i=1}^MM_{w,i}|\u|^2\J_i+\sum_{i,j=1}^Mc_{ij}\frac{\partial n_i}{\partial t}\grad{n_j} .
\end{eqnarray}
The fluxes $\q$ and $\J_i$ may rely on $\grad \frac{\mu_i}{T}$ and $\grad \frac{1}{T}$ according to Curie's Principle.  Onsager's  principle suggests that $\q$ and $\J_i$ have a linear dependent relationship  with $\grad \frac{\mu_i}{T}$ and $\grad \frac{1}{T}$ and moreover, this linear  relationship  shall be symmetric. So  there exists the linear, symmetric  relationship given in \eqref{eqDiffusionHeatFluxes}.
To obey the second law of thermodynamics,   the sum of  two related  non-conservative terms shall be non-negative, i.e.,
\begin{equation}
\q\cdot \grad\frac{1}{T}   -\sum_{i=1}^M\J_i\cdot\grad \frac{\mu_i+M_{w,i}gh}{T}\geq0,
\end{equation}
 which requires that the mobility matrix $\bm{\mathcal{L}}$ given in \eqref{eqDiffusionHeatFluxes} must be    positive definite or  positive semi-definite.  
 
 We rewrite \eqref{eqinitialTotalenergyEqn01}  as 
 \begin{eqnarray}\label{eqinitialTotalenergyEqn02}
 \frac{\partial  e_t}{\partial t}+\div\(\u  e_t+\bm\sigma\cdot\u\) =-\div\bm\phi_q ,
\end{eqnarray}
which leads to  the total energy conservation equation \eqref{eqTotalEnergyConservation} taking into account  \eqref{eqGenHeatFlux}.

Finally, the entropy equation \eqref{eqEntropyMultiGrad02} is reduced into the following form
\begin{eqnarray}\label{eqEntropyeqfinal}
 \frac{\partial s}{\partial t}+\div(\u s) 
 &=&-\div\frac{\q}{T}+\q\cdot \grad\frac{1}{T}  + \frac{1}{T}\bm\tau\(\u\):\grad\u\nonumber\\
 &&+\sum_{i=1}^M\div\(\frac{\mu_i+M_{w,i}gh}{T} \J_i\) -\sum_{i=1}^M\J_i\cdot\grad \frac{\mu_i+M_{w,i}gh}{T}\nonumber\\
 &=&-\frac{1}{T}\div \q  + \frac{1}{T}\bm\tau\(\u\):\grad\u +\sum_{i=1}^M\frac{\mu_i+M_{w,i}gh}{T}\div \J_i.
\end{eqnarray}

The above model equations are established such that the entropy production (non-conservative) terms in \eqref{eqEntropyeqfinal} are all non-negative, thereby obeying the second law of thermodynamics.

\subsection{The proof of the relation \eqref{eqRelationBetweenChPtlTempPres}}\label{secRelationBetweenChPtlTempPres}

Using the formulations of the pressure and chemical potentials,  we can deduce 
\begin{eqnarray}\label{eqPresChptlMultiGradproof01}
\sum_{i=1}^Mn_i\grad \mu_i-\grad p 
&=&\sum_{i=1}^Mn_i\grad\mu_i^b-\grad p_b-\sum_{i=1}^Mn_i\grad\(\sum_{j=1}^M\div\(c_{ij}\grad{n_j}\)\)\nonumber\\
 &&+\sum_{i,j=1}^M\grad\big(n_i\div\(c_{ij}\grad{n_j}\)\big)+\frac{1}{2} \sum_{i,j=1}^M\grad\(c_{ij} \nabla n_i\cdot\nabla n_j\)\nonumber\\
 &&- \frac{1}{2}\sum_{i,j,k=1}^N(\grad n_i)\frac{\partial c_{jk}}{\partial n_i} \nabla n_j\cdot\nabla n_k\nonumber\\
 &=&\gamma_b\grad T+\sum_{i,j=1}^M\(\div\(c_{ij}\grad{n_j}\)\)\grad n_i- \frac{1}{2}\sum_{i,j,k=1}^N(\grad n_i)\frac{\partial c_{jk}}{\partial n_i} \nabla n_j\cdot\nabla n_k\nonumber\\
 &&+\frac{1}{2} \sum_{i,j=1}^M\(\grad c_{ij}\) \(\nabla n_i\cdot\nabla n_j\)+\frac{1}{2} \sum_{i,j=1}^Mc_{ij}\grad\( \nabla n_i\cdot\nabla n_j\)\nonumber\\
 &=&\gamma_b\grad T+\gamma_\grad\grad{T}+\div\(\sum_{i,j=1}^Mc_{ij}\grad n_i\otimes\grad n_j\),
\end{eqnarray}
where  we have also used the following identity 
\begin{equation}\label{eqidentity}
\sum_{i,j=1}^M\(\grad{n_i}\)\div\(c_{ij}\grad{n_j}\)+\frac{1}{2}\sum_{i,j=1}^M c_{ij} \grad\(\nabla n_i\cdot\nabla n_j\) =\sum_{i,j=1}^M\div c_{ij}\(\nabla n_i\otimes\nabla n_j\).
  \end{equation}
Consequently,
we obtain the  relation \eqref{eqRelationBetweenChPtlTempPres} taking into account $s=-\gamma$.

\section{Entropy stable     numerical method}
We have already shown that the proposed model obeys the  first and  second laws of thermodynamics.  In this section, we focus on developing   the semi-implicit time marching scheme preserving  these laws of thermodynamics.  The energy conservation equation will be solved in the proposed scheme, and thus, the first law of thermodynamics can be naturally satisfied. The key effort put in designing the efficient  numerical scheme is how to preserve the second law of thermodynamics, which states that the total entropy can never decrease over time for an isolated system. Here, the  entropy stability of a numerical scheme means that it obeys the second law of thermodynamics  at the time-discrete level.

We will develop an entropy stable numerical method based on the use of an  intermediate velocity and the convex-concave splitting of  Helmholtz free energy density with respect to molar density and temperature. The convex-concave splitting approaches  usually leading  to the  energy-dissipation feature for numerical simulation of  the system with a constant temperature will be proved to result in the entropy stability of a numerical scheme for the  non-isothermal systems. A great challenge in designing the entropy-stable scheme is the very tightly coupling relationship among  molar density, energy (temperature)  and velocity, the treatment of which requires very careful  mathematical and physical observations.

 \subsection{Convex-concave splitting of  Helmholtz free energy density}
 
 We analyze the convex-concave splitting of  Helmholtz free energy density based on Peng-Robinson equation of state, which is   widely  used in the oil reservoir and chemical engineering.
The gradient term of Helmholtz free energy density is always convex with respect to molar densities.   For the bulk Helmholtz free energy density for  the multi-component mixture, there are a few approaches proposed in \cite{fan2017componentwise,kou2018nvtpc,kouandsun2017modeling}  to get a strict convex-concave splitting scheme.  Here,  we follow the approach proposed in \cite{kouandsun2017modeling} to introduce  a stabilization  term
 \begin{eqnarray}\label{eqHelmholtzEnergy_a0_02R}
 f_b^{\textnormal{stab}}(\n,T)=RT\sum_{i=1}^{M}n_i\(\ln n_i-1\)-RT\sum_{i=1}^Mn_i\ln\(1-b_in_i\),
\end{eqnarray}
which has a diagonal positive definite Hessian matrix, and thus is convex with respect to molar densities. Let us introduce a stabilization parameter $\theta\geq0$, and then we reformulate the bulk Helmholtz free energy density   as a sum of the convex part $f_b^{\textnormal{convex}}$ and the concave part $f_b^{\textnormal{concave}}$
\begin{subequations}\label{eqConvexConcaveHelmholtzEnergy}
\begin{equation}\label{eqConvexConcaveHelmholtzEnergy01}
    f_b(\n,T)= f_b^{\textnormal{convex}}(\n,T) +f_b^{\textnormal{concave}}(\n,T),
\end{equation}
\begin{equation}\label{eqConvexHelmholtzEnergy}
    f_b^{\textnormal{convex}}(\n,T)= f_b^{\textnormal{ideal}}(\n,T) +f_b^{\textnormal{repulsion}}(\n,T)+\theta f_b^{\textnormal{stab}}(\n,T),
\end{equation}
\begin{equation}\label{eqConcaveHelmholtzEnergy}
    f_b^{\textnormal{concave}}(\n,T)= f_b^{\textnormal{attraction}}(\n,T)-\theta f_b^{\textnormal{stab}}(\n,T) ,
\end{equation}
\end{subequations}
where the detailed formulations of $f_b^{\textnormal{ideal}}, ~f_b^{\textnormal{repulsion}}$ and $ f_b^{\textnormal{attraction}}$ can be found in Appendix.
The chemical potentials can be reformulated accordingly
\begin{eqnarray}\label{eqConvexChPtl}
    \mu_i^b(\n,T)= \mu_i^{b,\textnormal{convex}}(\n,T) +\mu_i^{b,\textnormal{concave}}(\n,T),
\end{eqnarray}
where 
\begin{equation}\label{eqDefConvexChPtl}
    \mu_i^{b,\textnormal{convex}}(\n,T)=\frac{\partial f_b^{\textnormal{convex}}(\n,T)}{\partial n_i}, ~~~\mu_i^{b,\textnormal{concave}}(\n,T)=\frac{\partial f_b^{\textnormal{concave}}(\n,T)}{\partial n_i}.
\end{equation}

We now turn to consider the convex-concave property of Helmholtz free energy density with respect to the temperature.  It is noted that $ \psi_i^p$ in the formulation of $f_b$ is  the molar heat capacity of  the ideal gas at the constant pressure. We  recall  the following thermodynamical relation  for  the ideal gas
\begin{eqnarray}\label{eqRelationHeatCapacity}
 \psi_i^p=R+ \psi_{i}^v,~~i=1,\cdots,M,
 \end{eqnarray}
where $R$ is the  universal gas constant and $\psi_{i}^v>0$ is the molar heat capacity of component $i$ at the constant volume for the ideal gas.
\begin{lem}\label{lemConvex}
The second derivative of the bulk Helmholtz free energy density $f_b$ with respective to the temperature  satisfies  
\begin{equation}\label{eqD2fbT}
\(\frac{\partial^2 f_b(\n,T)}{\partial T^2}\)_\n\leq0,
\end{equation}
and thus $f_b$ is concave with respect to the temperature.
\end{lem}
\begin{proof}
Let $n=\sum_{i=1}^Mn_i$. Fixing  molar densities, we calculate the second derivative of $f_b$ with respect to $T$ as
 \begin{eqnarray}\label{eqSecDerEntropy}
    \frac{\partial^2 f_b(\n,T)}{\partial T^2}&=&  \frac{nR}{T}-\sum_{i=1}^Mn_i\frac{\psi_{i}^p}{T} +\frac{na''(T)}{2\sqrt{2}b}\ln\(\frac{1+(1-\sqrt{2})b n}{1+(1+\sqrt{2})b n}\)\nonumber\\
    &=&  -\sum_{i=1}^Mn_i\frac{\psi_{i}^v}{T} +\frac{na''(T)}{2\sqrt{2}b}\ln\(\frac{1+(1-\sqrt{2})b n}{1+(1+\sqrt{2})b n}\),
\end{eqnarray}
where we have used the relation \eqref{eqRelationHeatCapacity}. The first term on the right-hand side of \eqref{eqSecDerEntropy}  is negative.  We note that the parameters $a_i$ and $b_i$ have defined in \eqref{eqaibi}. From the definition of $a(T)$ in \eqref{eqab}, we calculate
\begin{eqnarray*}
   a''(T)=\frac{1}{4}\sum_{i,j=1}^M y_i y_j(a_ia_j)^{-\frac{3}{2}}(1-k_{ij})A_{ij},
\end{eqnarray*}
\begin{eqnarray*}
   A_{ij}&=&\frac{m_ia_i^2a_j^2}{T\sqrt{TT_{c,i}}\(1+m_i(1-\sqrt{T_{r,i}})\)}+\frac{m_ja_i^2a_j^2}{T\sqrt{TT_{c,j}}\(1+m_j(1-\sqrt{T_{r,j}})\)}\nonumber\\
   &&+\frac{2m_im_ja_i^2a_j^2}{T\sqrt{T_{c,i}T_{c,j}}\(1+m_i(1-\sqrt{T_{r,i}})\)\(1+m_j(1-\sqrt{T_{r,j}})\)},
\end{eqnarray*}
where $T_{r,i}=T/T_{c,i}$, $T_{c,i}$ is the critical temperature and $m_i$ is the parameter given in \eqref{eqdefmi}.
It is obtained that $A_{ij}>0$ by the physical definitions of parameters, and thus we have $a''(T)>0$ taking into account   $k_{ij}<1$. Consequently, \eqref{eqD2fbT} is reached.
\end{proof}

For the gradient term of Helmholtz free energy density $f_\grad$, we take the constant influence parameters $c_{ij}$ in numerical tests  as usual practices in the dynamical van der Waals model. This choice makes the derivatives of $f_\grad$ with respect to the temperature zero when molar densities are fixed.  

It is easy to prove that Helmholtz free energy density based on van der Waals equation of state has  the property \eqref{eqD2fbT}, and can also be split into the sum of a convex part and a  concave part with respect to molar density. In the following,  we assume that  the influence parameters $c_{ij}$ may rely on the temperature but independent of molar densities and the convex-concave splitting can be reached  for the Helmholtz free energy density considered in this paper; more precisely, $f_b$ can be written as the form \eqref{eqConvexConcaveHelmholtzEnergy01} and $\(\frac{\partial^2 f(\n,T)}{\partial T^2}\)_\n\leq0$ always holds. 

\subsection{Semi-implicit time  scheme}
We now contruct the semi-implicit time marching scheme.  
We divide the  time interval $\mathcal{T}=(0,t_{f}]$, where $t_{f}>0$  into   $N$ subintervals $\mathcal{T}_{k}=(t_{k},t_{k+1}]$, where $t_{0}=0$ and $t_{N}=t_{f}$. The time step size is denoted as $\delta t_{k}=t_{k+1}-t_{k}$.  For a scalar function $v(t)$ or a vector function $\v(t)$, we denote by $v^{k}$ or $\v^{k}$ its approximation at the time $t_{k}$.

 For chemical potentials,    we treat  molar densities implicitly in their convex parts, while we apply the explicit treatment for their concave parts; meanwhile, the temperature is always implicit in both parts.  More precisely, the time discrete chemical potential of component $i$ at the $(k+1)$th time step is expressed as
\begin{subequations}\label{eqDiscretegeneralchemicalpotentialForm}
 \begin{equation}\label{eqDiscretegeneralchemicalpotentialForm01}
  \mu_i^{k+1}=\mu_i^{b,k+1}+\mu_{\grad,i}^{k+1}, 
  \end{equation}
   \begin{equation}\label{eqDiscretegeneralchemicalpotentialForm02}
  \mu_i^{b,k+1}=\mu_i^{b,\textnormal{convex}}(\n^{k+1},T^{k+1})+\mu_i^{b,\textnormal{concave}}(\n^{k},T^{k+1}), 
  \end{equation}
   \begin{equation}\label{eqDiscretegeneralchemicalpotentialForm03}
  \mu_{\grad,i}^{k+1}=-\sum_{j=1}^M\div\(c_{ij}^{k+1}\grad{n_j^{k+1}}\).
  \end{equation}
  \end{subequations}
  where $c_{ij}^{k+1}=c_{ij}(T^{k+1})$.
The relation parameters in the mass diffusion fluxes  and the  heat flux may rely on molar densities and temperature, and here, we take the following schemes 
\begin{subequations}\label{eqDiscreteDiffusionHeatFluxes}
\begin{equation}\label{eqDiscreteDiffusionHeatFluxes01}
\J_i^{k+1}=-\sum_{j=1}^M\mathcal{L}_{i,j}(\n^{k},T^{k})\grad \frac{\mu_j^{k+1}+M_{w,j}gh}{T^{k+1}} + \mathcal{L}_{i,M+1}(\n^{k},T^{k})\grad \frac{1}{T^{k+1}},
\end{equation}
\begin{equation}\label{eqDiscreteDiffusionHeatFluxes02}
\q^{k+1}=-\sum_{j=1}^M\mathcal{L}_{M+1,j}(\n^{k},T^{k})\grad \frac{\mu_j^{k+1}+M_{w,j}gh}{T^{k+1}} + \mathcal{L}_{M+1,M+1}(\n^{k},T^{k+1})\grad \frac{1}{T^{k+1}}.
\end{equation}
\end{subequations}
The above parameter matrix $\(\mathcal{L}_{i,j}\)_{i,j=1}^M$ still ensures the symmetry and  semi-positive definite property. 
 We introduce  an intermediate  velocity as
 \begin{eqnarray}\label{eqDecoupledDiscreteVelocityStar}
\u_\star^{k} = \u^k-\frac{\delta t_{k}}{\rho^k} \(\sum_{i=1}^Mn_i^{k}\grad \mu_i^{k+1}+ s^{k}\grad T^{k+1}-\rho^k\g\).
\end{eqnarray}
Based on the formulations of  $\u_\star^{k}$ and $\J_i^{k+1}$, we propose the following  semi-implicit  scheme  for the mass balance equation of component $i$
\begin{equation}\label{eqDiscreteMassEq}
\frac{ n_i^{k+1}-n_i^k}{\delta t_k}+\div\(n_i^{k}\u_\star^{k}\)+\div \J_i^{k+1}=0.
\end{equation}
The momentum balance equation is discretized as 
\begin{equation}\label{eqdiscreteMomentumConserveEq}
 \rho^k\(\frac{  \u^{k+1}-\u_\star^k}{\delta t_k}+ \u_\star^k\cdot\grad{ \u^{k+1}}\)
 +\sum_{i=1}^MM_{w,i}\J_i^{k+1}\cdot\grad\u^{k+1}=\div\bm{\tau}^{k+1},
\end{equation}
where 
\begin{equation}\label{eqDiscreteTotalStressB}
\bm\tau^{k+1}= \(\lambda^k\div\u^{k+1}\)\I+\eta^k\bm\varepsilon(\u^{k+1}).
\end{equation}

The total energy at the $k$th time step is expressed as
\begin{equation}\label{eqDiscreteTotalEnergy}
e_t^{k}=f(\n^k,T^k)+T^ks^k+\frac{1}{2}\rho^k|\u^k|^2+\rho^kgh.
\end{equation}
We also introduce the intermediate total energy as
\begin{equation}\label{eqDiscreteIntermediateTotalEnergy}
e_{t\star}^{k}=f(\n^{k+1},T^{k+1})+T^{k+1}s^k+\frac{1}{2}\rho^k|\u^{k+1}|^2+\rho^kgh.
\end{equation}
The pressure at the $(k+1)$th time step is calculated as
\begin{equation}\label{eqDiscreteGeneralPres}
p^{k+1}=\sum_{i=1}^Mn_i^k\mu_i^{k+1}-f(\n^{k+1},T^{k+1}).
\end{equation}
The total energy conservation equation has the following discrete form
\begin{equation}\label{eqDiscreteTotalEnergyConservation}
 \frac{ e_t^{k+1}-e_t^{k}}{\delta t_k} + \div\(\u_\star^{k} (e_{t\star}^{k}+p^{k+1})-\bm\tau^{k+1}\cdot\u^{k+1}\)=-\div\(\q^{k+1}-\bm\pi^{k+1}\) ,
\end{equation}
where \begin{equation}\label{eqDiscreteGenHeatFlux}
\bm\pi^{k+1} =  \sum_{i,j=1}^Mc_{ij}^{k+1}\frac{ n_i^{k+1}-n_i^k}{\delta t_k}\grad{n_j^{k+1}}-\frac{1}{2}\sum_{i=1}^MM_{w,i}|\u^{k+1}|^2\J_i^{k+1}  .
\end{equation}

\subsection{Unconditional entropy stability}

 We will derive an equality of entropy to prove that the proposed numerical scheme is unconditionally entropy stable. 
 
The entropy  equality will be deduced from the discrete total energy conservation equation through reducing  the  kinetic  energies and Helmholtz free energies, and for this purpose, we need first to prove the following lemmas on the variations of the  kinetic  energies and Helmholtz free energies.
\begin{lem}\label{lemKineticEnergiesVariation}
The variation of kinetic  energies between  the $k$th and $(k+1)$th  time steps is estimated as
 \begin{eqnarray}\label{eqKineticEnergiesVariation}
\frac{\rho^{k+1}|\u^{k+1}|^2-\rho^k|\u^{k}|^2}{2\delta t_{k}}&\leq&
-\frac{1}{2}\div\(\u_\star^{k}\rho^{k}|\u^{k+1}|^2\)-\frac{1}{2}\sum_{i=1}^MM_{w,i}\div\(|\u^{k+1}|^2\J_i^{k+1}\)\nonumber\\
&&-\u_\star^{k}\cdot\(n^{k}\grad \mu^{k+1}+ s^{k}\grad T^{k+1}-\rho^k\g\)\nonumber\\
&&+\div\(\bm{\tau}^{k+1}\cdot\u^{k+1}\)-\bm{\tau}^{k+1}:\grad\u^{k+1}.
\end{eqnarray}
\end{lem}
\begin{proof}
For the kinetic  energy at the $(k+1)$th time step and the intermediate kinetic energy, we derive their difference 
 \begin{align}\label{eqKineticEnergiesVariationProof01}
\frac{\rho^{k+1}|\u^{k+1}|^2-\rho^k|\u_\star^{k}|^2}{2}
&=\frac{1}{2}\rho^{k}\(|\u^{k+1}|^2-|\u_\star^{k}|^2\)+\frac{1}{2}\(\rho^{k+1}-\rho^{k}\)|\u^{k+1}|^2\nonumber\\
   &= \rho^{k}\(\u^{k+1}-\u_\star^{k}\)\cdot\u^{k+1}-\frac{1}{2}\rho^{k}|\u^{k+1}-\u_\star^{k}|^2\nonumber\\
   &~~~+\frac{1}{2}\(\rho^{k+1}-\rho^{k}\)|\u^{k+1}|^2.
\end{align}
The overall mass balance equation is deduced from the component mass balance equations \eqref{eqDiscreteMassEq} as
\begin{eqnarray}\label{eqDiscreteOverallMassBalanceEquation}
\frac{ \rho^{k+1}-\rho^{k}}{\delta t_{k}}+\div(\rho^{k}\u_\star^{k})+\sum_{i=1}^MM_{w,i}\div\J_i^{k+1}=0.
\end{eqnarray}
Substituting \eqref{eqdiscreteMomentumConserveEq} and \eqref{eqDiscreteOverallMassBalanceEquation} into \eqref{eqKineticEnergiesVariationProof01} yields
\begin{eqnarray}\label{eqKineticEnergiesVariationProof02}
\frac{\rho^{k+1}|\u^{k+1}|^2-\rho^k|\u_\star^{k}|^2}{2\delta t_{k}}&=&-\(\rho^{k}\u_\star^{k}\cdot\grad\u^{k+1}+\sum_{i=1}^MM_{w,i}\J_i^{k+1}\cdot\grad\u^{k+1}\)\cdot\u^{k+1}\nonumber\\
&&+\u^{k+1}\cdot\div\bm{\tau}^{k+1}\nonumber\\
&&-\frac{1}{2}|\u^{k+1}|^2\(\div(\rho^{k}\u_\star^{k})+\sum_{i=1}^MM_{w,i}\div\J_i^{k+1}\)\nonumber\\
&&-\frac{1}{2\delta t_{k}}\rho^{k}|\u^{k+1}-\u_\star^{k}|^2\nonumber\\
&=&-\frac{1}{2}\div\(\u_\star^{k}\rho^{k}|\u^{k+1}|^2\)-\frac{1}{2}\sum_{i=1}^MM_{w,i}\div\(|\u^{k+1}|^2\J_i^{k+1}\)\nonumber\\
&&+\div\(\bm{\tau}^{k+1}\cdot\u^{k+1}\)-\bm{\tau}^{k+1}:\grad\u^{k+1}\nonumber\\
&&-\frac{1}{2\delta t_{k}}\rho^{k}|\u^{k+1}-\u_\star^{k}|^2.
\end{eqnarray}
Using the definition of  the intermediate kinetic energy, we deduce that 
\begin{eqnarray}\label{eqKineticEnergiesVariationProof03}
\frac{\rho^k|\u_\star^{k}|^2-\rho^k|\u^{k}|^2}{2}&=&\rho^k\(\u_\star^{k}-\u^k\)\cdot \u_\star^{k}-\frac{1}{2}\rho^{k}|\u_\star^{k}-\u^{k}|^2\nonumber\\
&=&-\delta t_{k} \u_\star^{k}\cdot\(n^{k}\grad \mu^{k+1}+ s^{k}\grad T^{k+1}-\rho^k\g\)\nonumber\\
&&-\frac{1}{2}\rho^{k}|\u_\star^{k}-\u^{k}|^2.
\end{eqnarray}
We combine \eqref{eqKineticEnergiesVariationProof02} and \eqref{eqKineticEnergiesVariationProof03} to get
\begin{eqnarray}\label{eqKineticEnergiesVariationProof04}
\frac{\rho^{k+1}|\u^{k+1}|^2-\rho^k|\u^{k}|^2}{2\delta t_{k}}&=&
-\frac{1}{2}\div\(\u_\star^{k}\rho^{k}|\u^{k+1}|^2\)-\frac{1}{2}\sum_{i=1}^MM_{w,i}\div\(|\u^{k+1}|^2\J_i^{k+1}\)\nonumber\\
&&-\u_\star^{k}\cdot\(n^{k}\grad \mu^{k+1}+ s^{k}\grad T^{k+1}-\rho^k\g\)\nonumber\\
&&+\div\(\bm{\tau}^{k+1}\cdot\u^{k+1}\)-\bm{\tau}^{k+1}:\grad\u^{k+1}\nonumber\\
&&-\frac{1}{2\delta t_{k}}\rho^{k}|\u^{k+1}-\u_\star^{k}|^2-\frac{1}{2\delta t_{k}}\rho^{k}|\u_\star^{k}-\u^{k}|^2,
\end{eqnarray}
which yields \eqref{eqKineticEnergiesVariation}.
  \end{proof}

\begin{lem}\label{lemFreeEnergyVariation}
The variation of Helmholtz free energies between  the $k$th and $(k+1)$th  time steps satisfy  the following inequality
 \begin{eqnarray}\label{eqFreeEnergyVariation}
\frac{ f^{k+1}-f^k}{\delta t_k}
 &\leq&-\div\(\u_\star^kf^{k+1}\)-\div\(\u^k_\star p^{k+1}\)-s^k\frac{ T^{k+1}-T^k}{\delta t_k}\nonumber\\
&&
 +\frac{1}{\delta t_k}\sum_{i,j=1}^M\div \(c_{ij}^{k+1}  \(n_i^{k+1}-n_i^k\)\grad n_j^{k+1}\)\nonumber\\
&&-\sum_{i=1}^M\mu_{i}^{k+1}\div\J_i^{k+1}+\sum_{i=1}^M\u^k_\star \cdot n_i^k\grad\mu_{i}^{k+1},
 \end{eqnarray}
 where we denote $f^k=f(\n^k,T^k)$ and $s^k=s(\n^{k},T^k)$.
\end{lem}
\begin{proof}
For the sake of notation simplification, we denote $f_b^k=f_b(\n^{k},T^k)$ and $f_\grad^k=f_\grad(\n^{k},T^k)$. Furthermore, we separate the pressure $p^{k+1}$ as $p^{k+1}=p_b^{k+1}+p_\grad^{k+1}$, where
\begin{equation*}
p_b^{k+1}=\sum_{i=1}^Mn_i^k\mu_i^{b,k+1}-f_b^{k+1},~~~
p_\grad^{k+1}=\sum_{i=1}^Mn_i^k\mu_{\grad,i}^{k+1}-f_\grad^{k+1}.
\end{equation*}
 Let  us denote $\gamma_b^k=\gamma_b(\n^k,T^k)$ and $\gamma_\grad^k=\gamma_\grad(\n^k,T^k)$.  Applying  the convex-concave properties and   component mass balance equations, we  deduce that
\begin{eqnarray}\label{eqFreeEnergyVariationProof01}
 \frac{ f_b^{k+1}-f_b^k}{\delta t_k}
 &=&\frac{ f_b(\n^{k+1},T^{k+1})-f_b(\n^{k},T^{k+1})}{\delta t_k}+\frac{ f_b(\n^{k},T^{k+1})-f_b(\n^{k},T^k)}{\delta t_k}
 \nonumber\\
 &\leq&\sum_{i=1}^M\mu_i^{b,k+1}\frac{ n_i^{k+1}-n_i^k}{\delta t_k}+\gamma_b^k\frac{ T^{k+1}-T^k}{\delta t_k}
 \nonumber\\
 &\leq&-\sum_{i=1}^M\mu_i^{b,k+1}\(\div(n_i^k\u^k_\star)+\div\J_i^{k+1}\)+\gamma_b^k\frac{ T^{k+1}-T^k}{\delta t_k}\nonumber\\
 &\leq&-\sum_{i=1}^M\(\div(\u_\star^kn_i^k\mu_i^{b,k+1})-\u^k_\star\cdot n_i^k\grad\mu_i^{b,k+1}\)\nonumber\\
 &&-\sum_{i=1}^M\mu_i^{b,k+1}\div\J_i^{k+1}+\gamma_b^k\frac{ T^{k+1}-T^k}{\delta t_k}\nonumber\\
 &\leq&-\div(f_b^{k+1}\u_\star^k)-\div(\u_\star^kp_b^{k+1})-\sum_{i=1}^M\u^k_\star\cdot n_i^k\grad\mu_i^{b,k+1}\nonumber\\
 &&-\sum_{i=1}^M\mu_i^{b,k+1}\div\J_i^{k+1}+\gamma_b^k\frac{ T^{k+1}-T^k}{\delta t_k}.
\end{eqnarray}
The variation of    gradient term of Helmholtz free energy density can be derived as
\begin{eqnarray}\label{eqFreeEnergyVariationProof02}
\frac{ f_\grad^{k+1}-f_\grad^k}{\delta t_k}
 &=&\frac{1}{2\delta t_k}\sum_{i,j=1}^M\(c_{ij}^{k+1} \nabla n_i^{k+1}\cdot\nabla n_j^{k+1}-c_{ij}^{k} \nabla n_i^{k}\cdot\nabla n_j^{k}\)\nonumber\\
 &=&\frac{1}{2\delta t_k}\sum_{i,j=1}^M\(c_{ij}^{k+1}-c_{ij}^{k}\) \nabla n_i^{k}\cdot\nabla n_j^{k}\nonumber\\
 &&+\frac{1}{2\delta t_k}\sum_{i,j=1}^Mc_{ij}^{k+1} \(\nabla n_i^{k+1}\cdot\nabla n_j^{k+1}-\nabla n_i^{k}\cdot\nabla n_j^{k}\)\nonumber\\
 &\leq&\gamma_\grad^k\frac{ T^{k+1}-T^k}{\delta t_k}
 +\frac{1}{\delta t_k}\sum_{i,j=1}^Mc_{ij}^{k+1} \grad \(n_i^{k+1}-n_i^k\)\cdot\grad n_j^{k+1}\nonumber\\
 &\leq&\gamma_\grad^k\frac{ T^{k+1}-T^k}{\delta t_k}
 +\frac{1}{\delta t_k}\sum_{i,j=1}^M\div c_{ij}^{k+1} \(n_i^{k+1}-n_i^k\)\grad n_j^{k+1}\nonumber\\
&&+\sum_{i=1}^M\frac{n_i^{k+1}-n_i^k}{\delta t_k}\mu_{\grad,i}^{k+1}\nonumber\\
 &\leq&\gamma_\grad^k\frac{ T^{k+1}-T^k}{\delta t_k}
 +\frac{1}{\delta t_k}\sum_{i,j=1}^M\div c_{ij}^{k+1}  \(n_i^{k+1}-n_i^k\)\grad n_j^{k+1}\nonumber\\
&&-\sum_{i=1}^M\(\div(n_i^k\u^k_\star)+\div\J_i^{k+1}\)\mu_{\grad,i}^{k+1}\nonumber\\
 &\leq&\gamma_\grad^k\frac{ T^{k+1}-T^k}{\delta t_k}
 +\frac{1}{\delta t_k}\sum_{i,j=1}^M\div c_{ij}^{k+1} \(n_i^{k+1}-n_i^k\)\grad n_j^{k+1}\nonumber\\
&&-\sum_{i=1}^M\mu_{\grad,i}^{k+1}\div\J_i^{k+1}-\sum_{i=1}^M\div\(\u^k_\star n_i^k\mu_{\grad,i}^{k+1}\)+\sum_{i=1}^M\u^k_\star \cdot n_i^k\grad\mu_{\grad,i}^{k+1}\nonumber\\
 &\leq&\gamma_\grad^k\frac{ T^{k+1}-T^k}{\delta t_k}
 +\frac{1}{\delta t_k}\sum_{i,j=1}^M\div c_{ij}^{k+1} \(n_i^{k+1}-n_i^k\)\grad n_j^{k+1}\nonumber\\
&&-\sum_{i=1}^M\mu_{\grad,i}^{k+1}\div\J_i^{k+1}-\div\(\u^k_\star p_{\grad}^{k+1}\)\nonumber\\
&&+\sum_{i=1}^M\u^k_\star \cdot n_i^k\grad\mu_{\grad,i}^{k+1}
-\div\(\u_\star^kf_\grad(\n^{k+1},T^k)\).
 \end{eqnarray}
Combining \eqref{eqFreeEnergyVariationProof01} and \eqref{eqFreeEnergyVariationProof02} and taking into account $s=-\gamma$, we obtain the inequality \eqref{eqFreeEnergyVariation}.
\end{proof}

We now prove that the proposed semi-implicit scheme  obeys the   laws of thermodynamics; that is, it is unconditionally entropy stable. 
\begin{thm}\label{thmDiscreteEntropyStability}
The proposed semi-implicit scheme   satisfies  the second law of thermodynamics in the sense of 
\begin{eqnarray}\label{eqDiscreteEntropyStability01}
  \frac{s^{k+1}-s^k}{\delta t_k}&\geq&-\div\frac{\q^{k+1}}{T^{k+1}}-\div(\u_\star^ks^k)+\sum_{i=1}^M\div\(\J_i^{k+1}\frac{\mu_{i}^{k+1}+M_{w,i}gh}{T^{k+1}}\)\nonumber\\
  &&+\q^{k+1}\cdot\grad\frac{1}{T^{k+1}}-\sum_{i=1}^M\J_i^{k+1}\cdot\grad\frac{\mu_{i}^{k+1}+M_{w,i}gh}{T^{k+1}}\nonumber\\
  &&+\frac{1}{T^{k+1}}\( \lambda^k|\div\u^{k+1}|^2+\frac{1}{2}\eta^k|\bm\varepsilon(\u^{k+1})|^2\),
 \end{eqnarray}
 where  $s^k=s(\n^k,T^k)$.
 Moreover, for an isolated system, we have
\begin{eqnarray}\label{eqDiscreteEntropyStability02}
   \frac{ \mathcal{S}^{k+1}-\mathcal{S}^{k}}{\delta t_{k}}&\geq&\(\q^{k+1},\grad\frac{1}{T^{k+1}}\)-\sum_{i=1}^M\(\J_i^{k+1},\grad\frac{\mu_{i}^{k+1}+M_{w,i}gh}{T^{k+1}}\)\nonumber\\
   &&+\( \lambda^k|\div\u^{k+1}|^2+\frac{1}{2}\eta^k|\bm\varepsilon(\u^{k+1})|^2,\frac{1}{T^{k+1}}\)\geq0,
 \end{eqnarray}
  where $\mathcal{S}^k=\int_\Omega s^kd\x$.
\end{thm}
\begin{proof}
 We note that $e_t^k=f^k+T^ks^k+\frac{1}{2}\rho^k|\u^k|^2+\rho^kgh$. We have already deduced the variations of the kinetic energy and Helmholtz free energy with time steps, and now we deduce the variation of gravity potential energy with time steps
\begin{eqnarray}\label{eqVariationgravitypotentialenergy}
  gh\frac{\rho^{k+1}-\rho^k}{\delta t_k}&=&-gh\(\div(\rho^{k}\u_\star^{k})+\sum_{i=1}^MM_{w,i}\div\J_i^{k+1}\)\nonumber\\
  &=&-\div(\rho^{k}gh\u_\star^{k})-\rho^{k}\u_\star^{k}\cdot\g-\sum_{i=1}^MM_{w,i}gh\div\J_i^{k+1}.
 \end{eqnarray}
Using \eqref{eqKineticEnergiesVariation}, \eqref{eqFreeEnergyVariation} and \eqref{eqVariationgravitypotentialenergy}, and taking into account 
 \begin{equation}\label{eqDiscreteEntropyStabilityProof01}
 T^{k+1}s^{k+1}-T^ks^k=T^{k+1}\(s^{k+1}-s^k\)+s^k(T^{k+1}-T^k),
 \end{equation}
  we reduce the  kinetic  energies,  Helmholtz free energies and gravity potential energies from the total energy equation to derive  that
 \begin{eqnarray}\label{eqDiscreteEntropyStabilityProof02}
  T^{k+1}\frac{s^{k+1}-s^k}{\delta t_k}&\geq&-\div(\q^{k+1}+\u_\star^kT^{k+1}s^k)+\sum_{i=1}^M\(\mu_{i}^{k+1}+M_{w,i}gh\)\div\J_i^{k+1}\nonumber\\
  &&+\u_\star^{k}\cdot s^{k}\grad T^{k+1}+\bm{\tau}^{k+1}:\grad\u^{k+1}.
 \end{eqnarray}
 Furthermore, we reduce \eqref{eqDiscreteEntropyStabilityProof02} as
 \begin{eqnarray}\label{eqDiscreteEntropyStabilityProof03}
  \frac{s^{k+1}-s^k}{\delta t_k}&\geq&-\frac{1}{T^{k+1}}\div\q^{k+1}-\div(\u_\star^ks^k)\nonumber\\
  &&+\sum_{i=1}^M\frac{\mu_{i}^{k+1}+M_{w,i}gh}{T^{k+1}}\div\J_i^{k+1}+\frac{1}{T^{k+1}}\bm{\tau}^{k+1}:\grad\u^{k+1},
   \end{eqnarray}
 which leads to \eqref{eqDiscreteEntropyStability01}.
 
 We now analyze the terms in the  right-hand side of \eqref{eqDiscreteEntropyStability01}.  The first three terms are the conservative terms, which do not produce the entropy, while the rest terms are    the entropy production terms, which shall be non-negative according to the second law of thermodynamics.  The non-negativity of the fourth  and fifth terms is true due to the choice principles of $\q^{k+1}$ and $\J_i^{k+1}$. The last term is obviously non-negative. Consequently, the proposed scheme obeys the second law of thermodynamics. 
 For   the isolated system,  the homogeneous Neumann boundary conditions are applied (that is, all terms vanish on the boundary when integration by parts is applied), and thus  the inequality \eqref{eqDiscreteEntropyStability02} can be obtained by integrating \eqref{eqDiscreteEntropyStability01} over the domain $\Omega$.
\end{proof}

\section{Numerical tests}

In this section, we will apply the proposed method  to simulate  non-isothermal multi-component two-phase   flow problems.  A binary   mixture composed of methane (C$_1$) and pentane (C$_5$) is located in   a square  domain $\Omega$ with the length $20$~nm. To validate the proposed method, we simulate the dynamics of the isolated systems, showing the entropy increase with time steps as proved in \eqref{eqDiscreteEntropyStability02}.  The mass-average velocity is employed with the diffusion mobility formulations given by  \eqref{eqMultiCompononentMassConserveDiffusionChoiceB}  taking  the mass-diffusion coefficients $\mathscr{D}_{12}=\mathscr{D}_{21}=10^{-8}$ m$^2$/s. The volumetric  viscosity and    the shear viscosity are taken equal to zero. The heat diffusion coefficient is taken as $\mathcal{K}=n\times10^{-3}$ J/(m$\cdot$s$\cdot$K), where $n=\sum_{i=1}^Mn_i$. The initial liquid-phase molar densities      are $n_{\textnormal{C}_1}^L=6.8663$ kmol/m$^3$ and $n_{\textnormal{C}_5}^L=4.7915$ kmol/m$^3$ respectively, while the initial gas-phase molar densities   are $n_{\textnormal{C}_1}^G=7.4302$ kmol/m$^3$ and $n_{\textnormal{C}_5}^G=0.6736$ kmol/m$^3$  respectively. The gravity effect is ignored. The parameter $\theta=0$ is taken in \eqref{eqConvexConcaveHelmholtzEnergy}. The rest physical parameters can be found in Appendix.
We use a  uniform  rectangular mesh with $40\times40$ elements.  We employ the cell-centered finite difference method and the upwind scheme   to discretize  the mass balance equations and total energy conservation equation,  and  the finite volume method  on the staggered mesh \cite{Tryggvason2011book}   for the momentum balance equation. These spatial discretization methods  can be   equivalent to the  special mixed finite element methods with specified  quadrature rules \cite{arbogast1997mixed,Girault1996Mac}.

In this example,  a square-shaped  droplet is located in   the center of the domain  at the initial time.   The initial molar density distributions of  C$_1$ and C$_5$ are depicted  in  Figures \ref{IsolatedSysC1andC5MolarDensityC1}(a) and \ref{IsolatedSysC1andC5MolarDensityC5}(a) respectively. The initial temperature is uniformly  taken equal to 310 K.   The time step size is taken as $10^{-12}$ s, and 60  time steps are simulated. 

The total entropy  profiles  with time steps are depicted in Figure \ref{IsolatedSysC1andC5TotalEntropy}(a), while  Figure \ref{IsolatedSysC1andC5TotalEntropy}(b)   is a zoom-in plot of Figure \ref{IsolatedSysC1andC5TotalEntropy}(a)  in the later time steps. It is  shown   that the total entropy  always  increase in the simulation process,  and    as a result, the proposed method can preserve the second law of thermodynamics.

  Figures \ref{IsolatedSysC1andC5MolarDensityC1} and \ref{IsolatedSysC1andC5MolarDensityC5} illustrate  the dynamical  process  of  molar densities of  C$_1$ and C$_5$ at different time steps, while  Figure \ref{IsolatedSysC1andC5Temperature} illustrates the temperature contours at different time steps. Moreover, in Figure \ref{IsolatedSysC1andC5Velocity}, we show  the fluid motion   including  the velocity field and  magnitudes of both velocity components. It   is clearly observed that the  initially square-shaped droplet is gradually changing   to a circle, and subsequently, the mixture vapor  is partially cooled and condenses into the liquid phase due to the lower temperature in the droplet region.

\begin{figure}
           \centering \subfigure[]{
            \begin{minipage}[b]{0.45\textwidth}
            \centering
             \includegraphics[width=0.95\textwidth,height=2in]{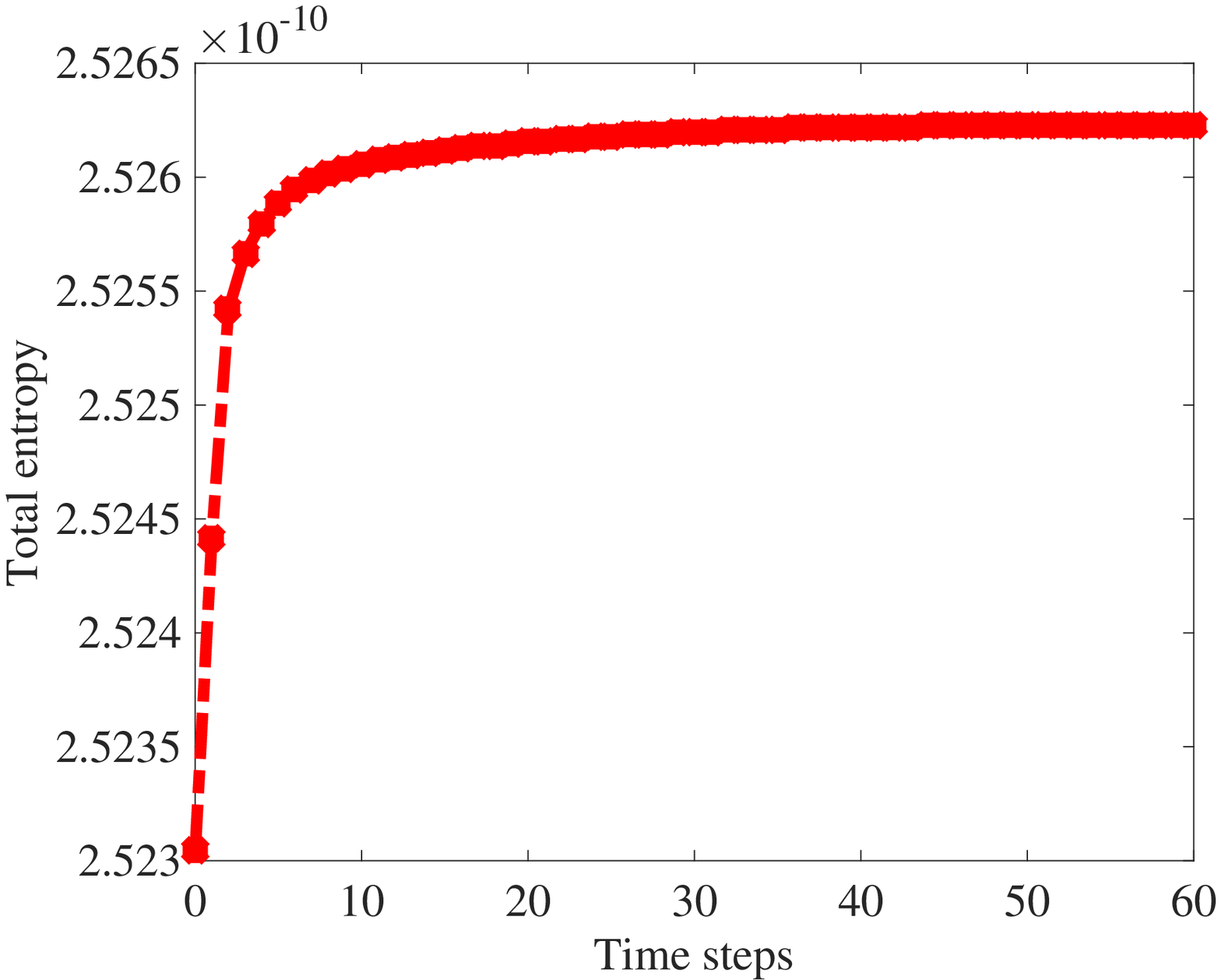}
            \end{minipage}
            }
            \centering \subfigure[]{
            \begin{minipage}[b]{0.45\textwidth}
            \centering
             \includegraphics[width=0.95\textwidth,height=2in]{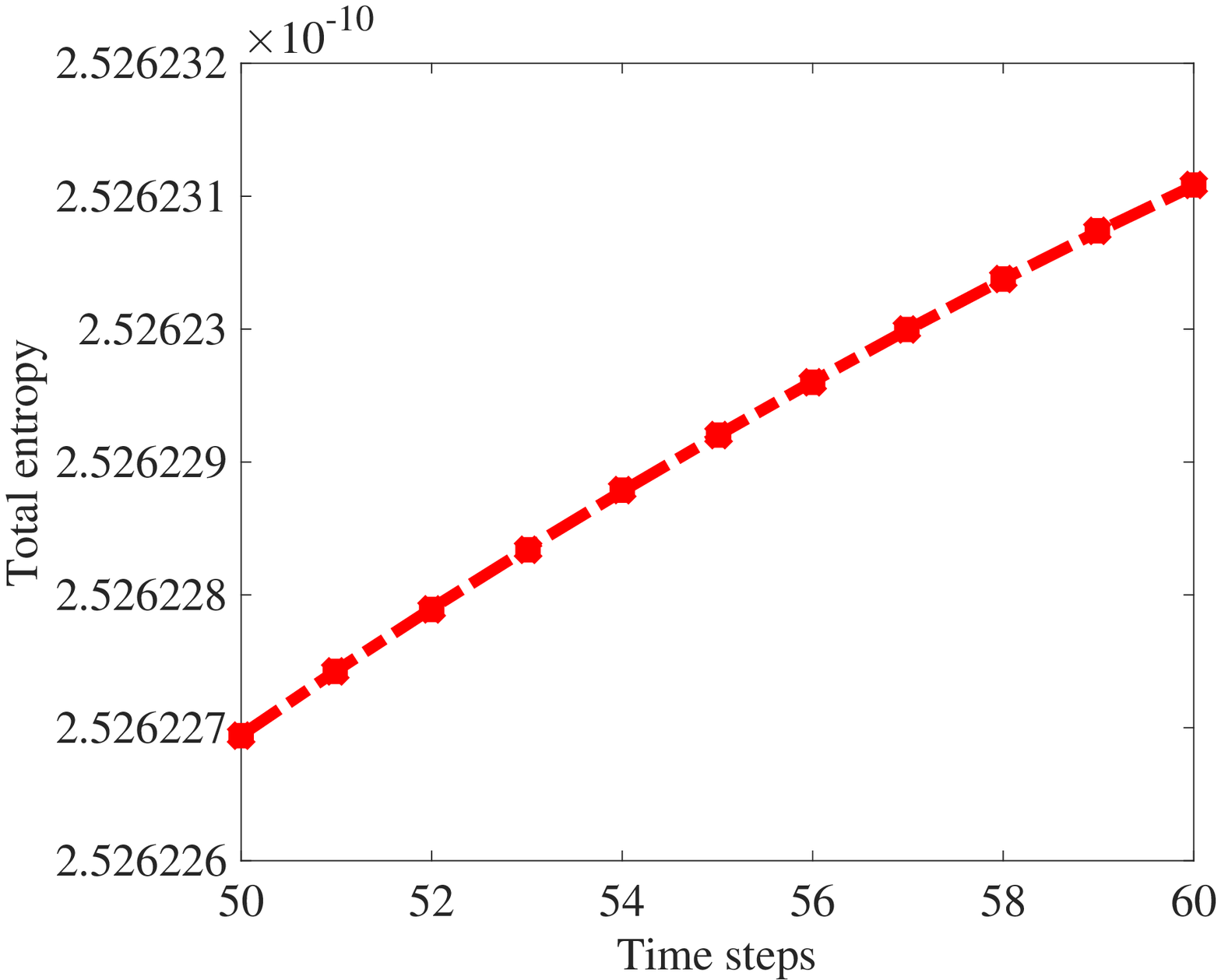}
            \end{minipage}
            }
           \caption{Example 1: the  total entropy profile with time steps.}
            \label{IsolatedSysC1andC5TotalEntropy}
 \end{figure}

\begin{figure}
            \centering \subfigure[C$_1$ at the initial time]{
            \begin{minipage}[b]{0.3\textwidth}
               \centering
             \includegraphics[width=\textwidth,height=0.9\textwidth]{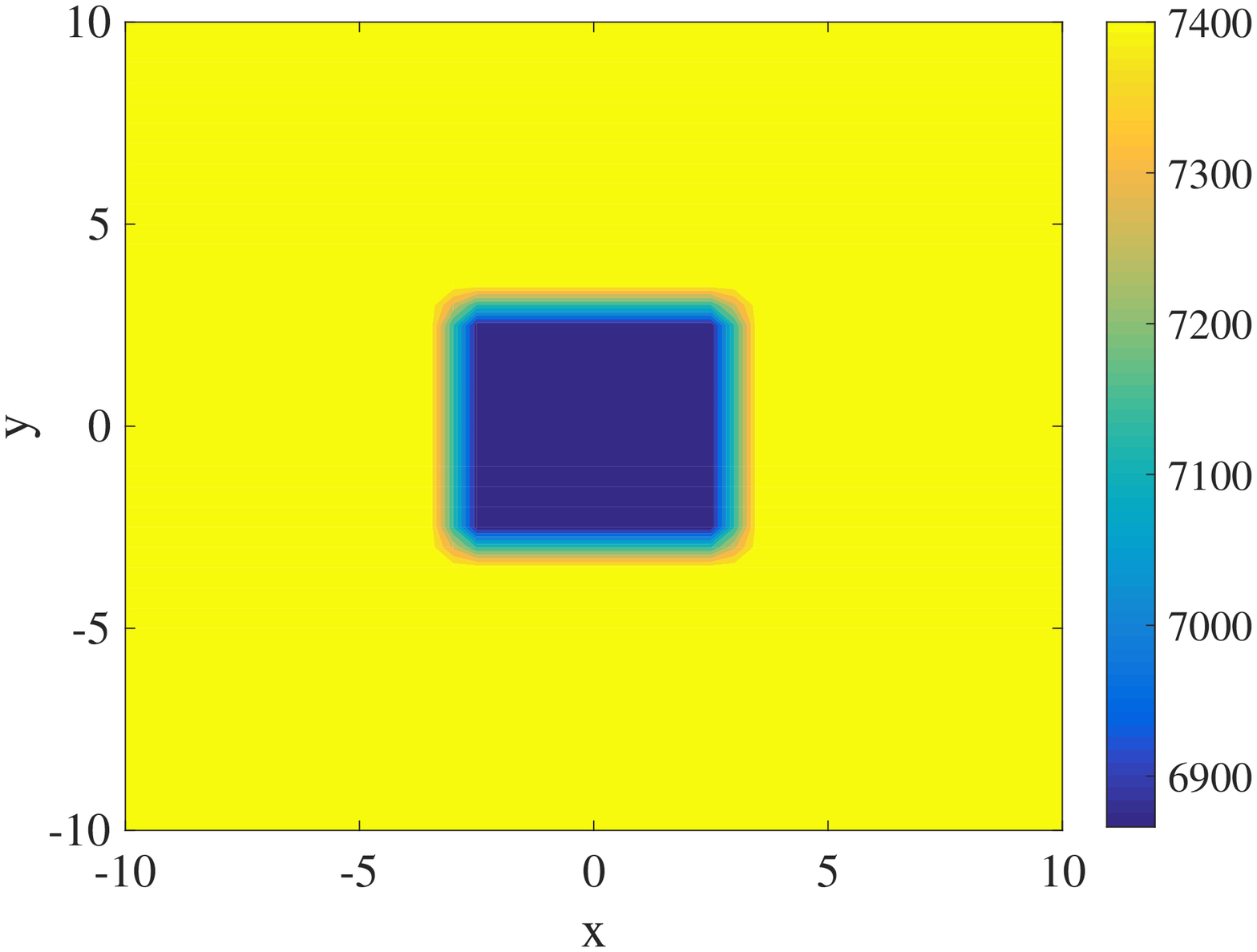}
            \end{minipage}
            }
            \centering \subfigure[C$_1$ at the 30th time step]{
            \begin{minipage}[b]{0.3\textwidth}
            \centering
             \includegraphics[width=\textwidth,height=0.9\textwidth]{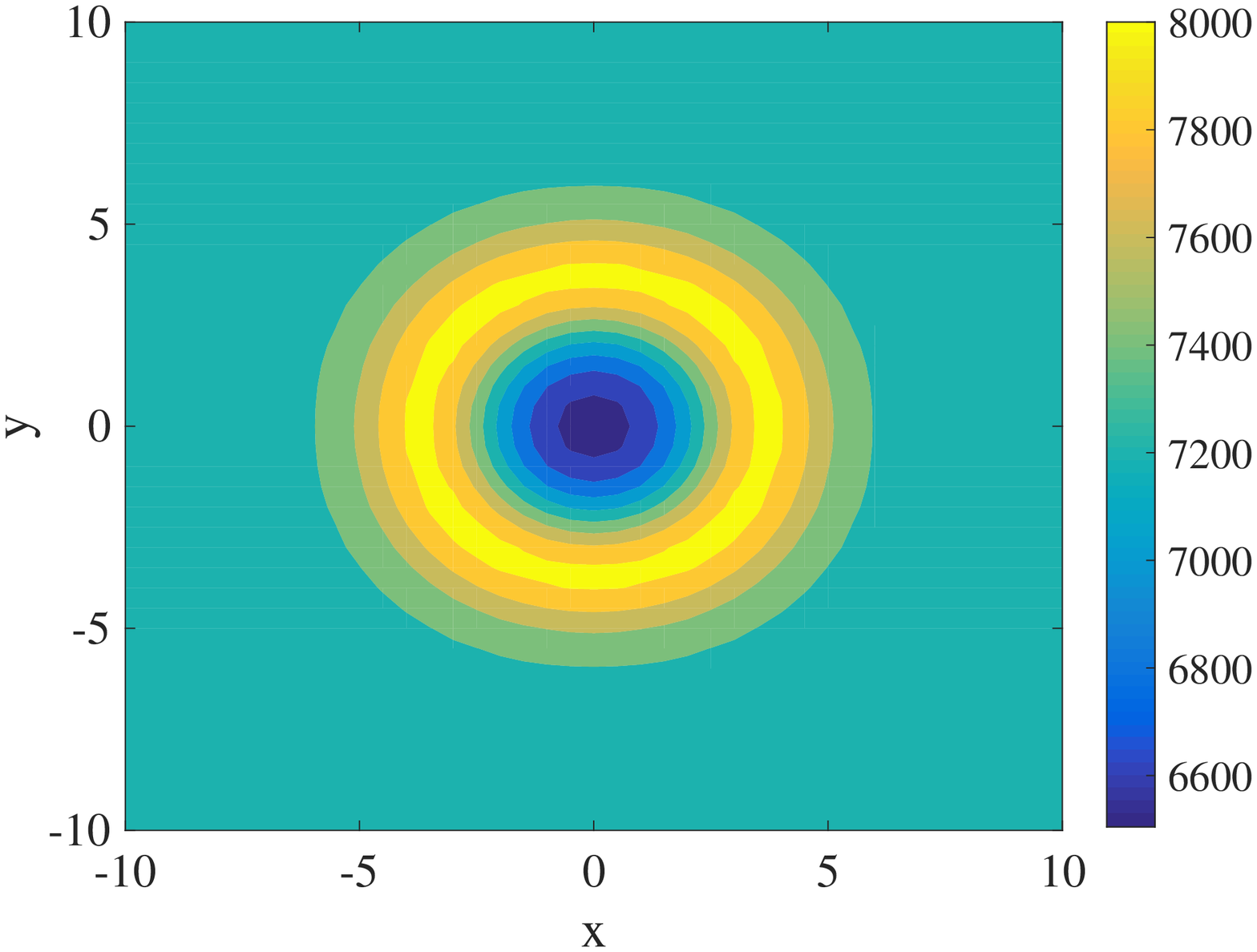}
            \end{minipage}
            }
           \centering \subfigure[C$_1$ at the 60th time step]{
            \begin{minipage}[b]{0.3\textwidth}
               \centering
             \includegraphics[width=\textwidth,height=0.9\textwidth]{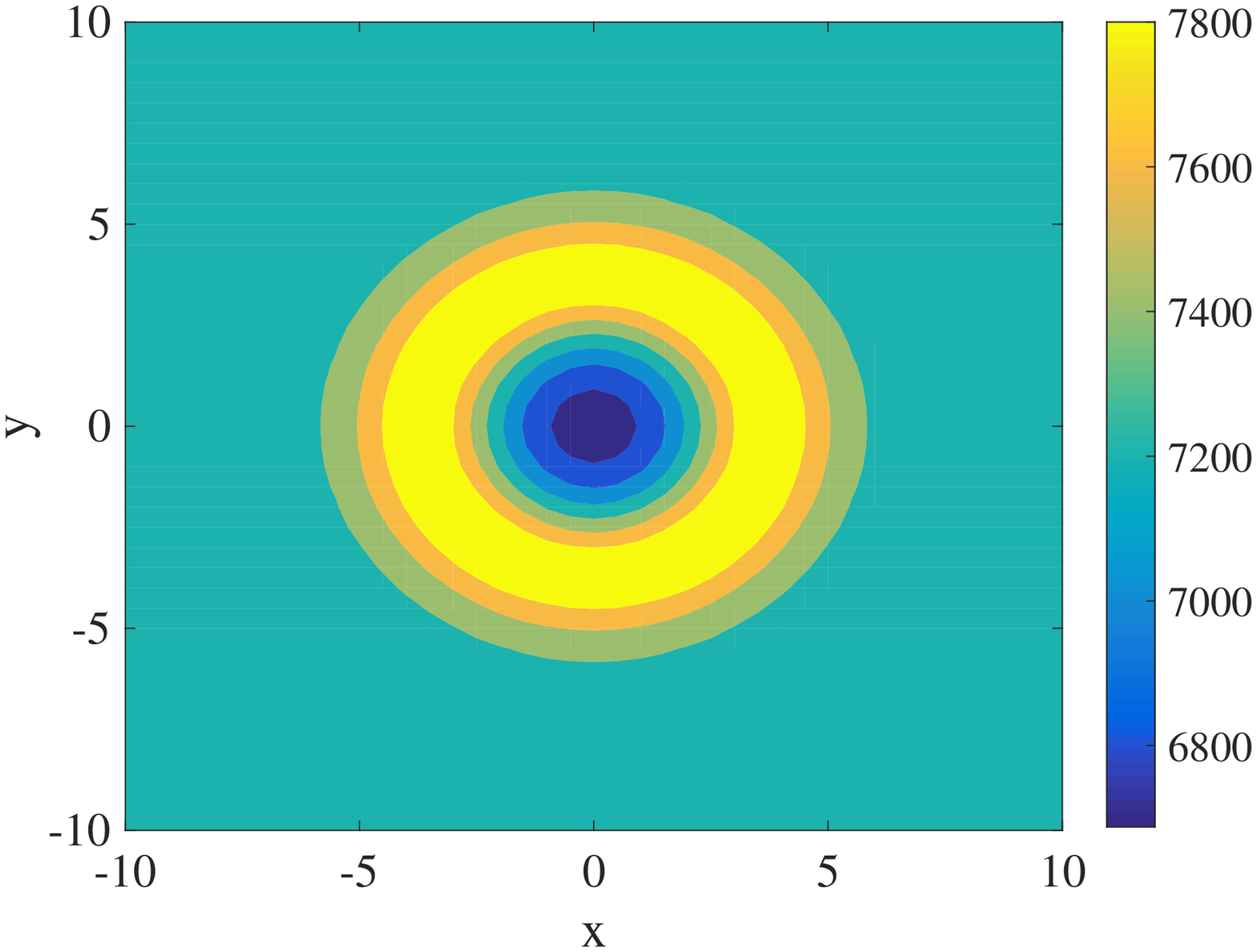}
            \end{minipage}
            }
            \caption{Example 1:  molar densities of C$_1$    at  different time steps.}
            \label{IsolatedSysC1andC5MolarDensityC1}
 \end{figure}
 
\begin{figure}
            \centering \subfigure[C$_5$ at the the initial time]{
            \begin{minipage}[b]{0.3\textwidth}
               \centering
             \includegraphics[width=\textwidth,height=0.9\textwidth]{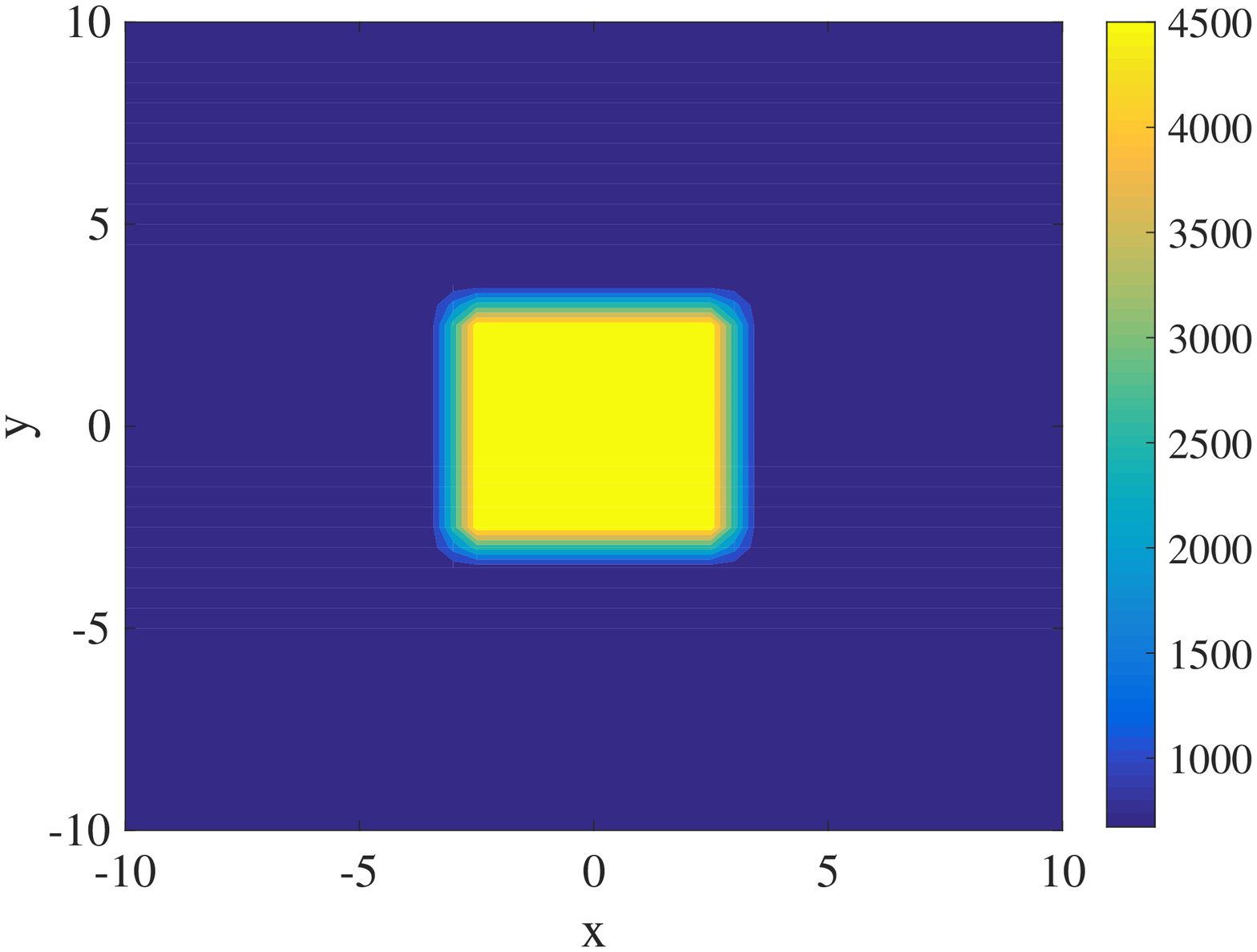}
            \end{minipage}
            }
            \centering \subfigure[C$_5$ at the 30th time step]{
            \begin{minipage}[b]{0.3\textwidth}
            \centering
             \includegraphics[width=\textwidth,height=0.9\textwidth]{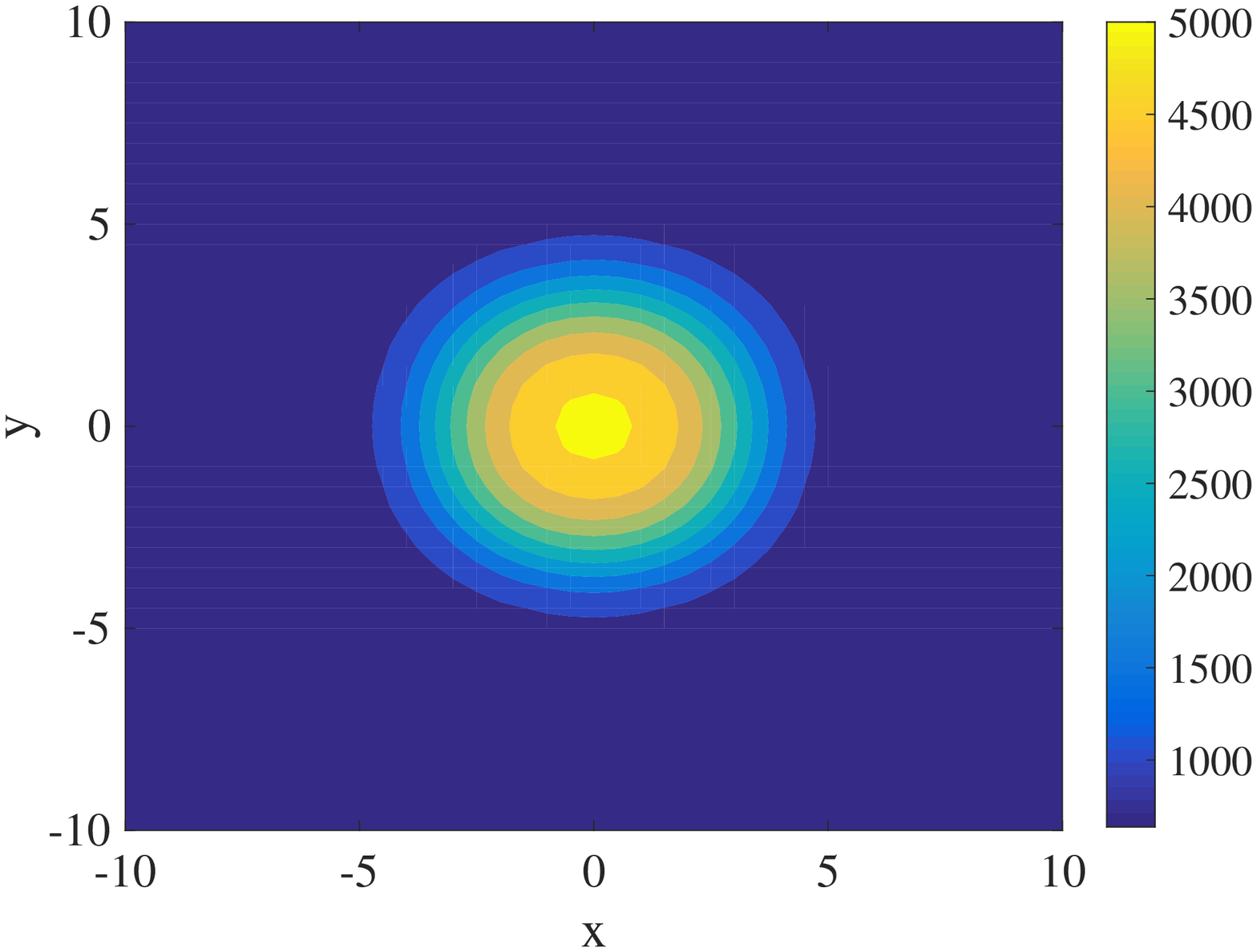}
            \end{minipage}
            }
           \centering \subfigure[C$_5$ at the 60th time step]{
            \begin{minipage}[b]{0.3\textwidth}
               \centering
             \includegraphics[width=\textwidth,height=0.9\textwidth]{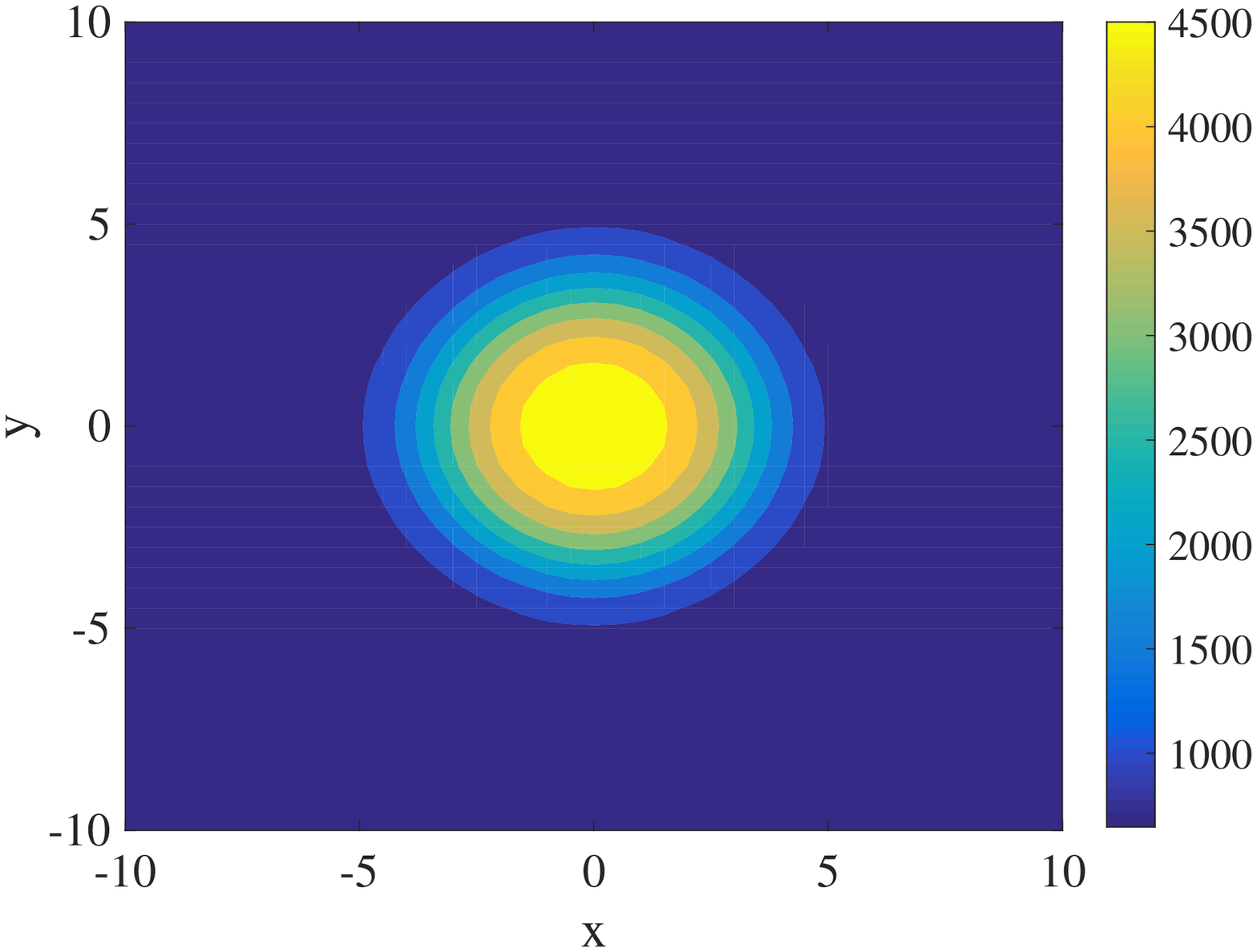}
             \end{minipage}
            }
            \caption{Example 1:  molar densities of   C$_5$  at  different time steps.}
            \label{IsolatedSysC1andC5MolarDensityC5}
 \end{figure}

\begin{figure}
           \centering \subfigure[$T$ at the10th time step]{
            \begin{minipage}[b]{0.3\textwidth}
            \centering
             \includegraphics[width=0.95\textwidth,height=0.9\textwidth]{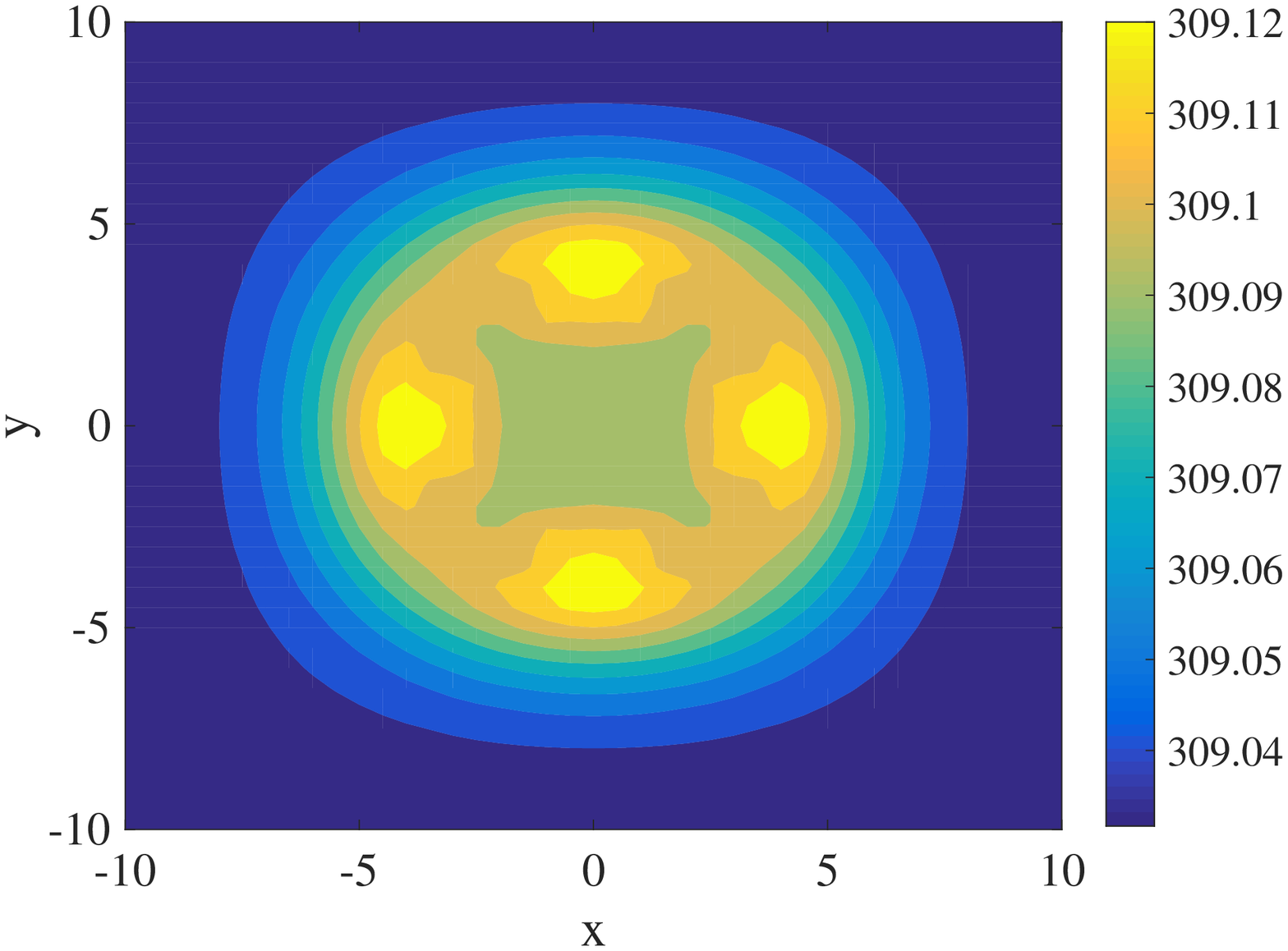}
            \end{minipage}
            }
           \centering \subfigure[$T$ at the 30th time step]{
            \begin{minipage}[b]{0.3\textwidth}
            \centering
             \includegraphics[width=0.95\textwidth,height=0.9\textwidth]{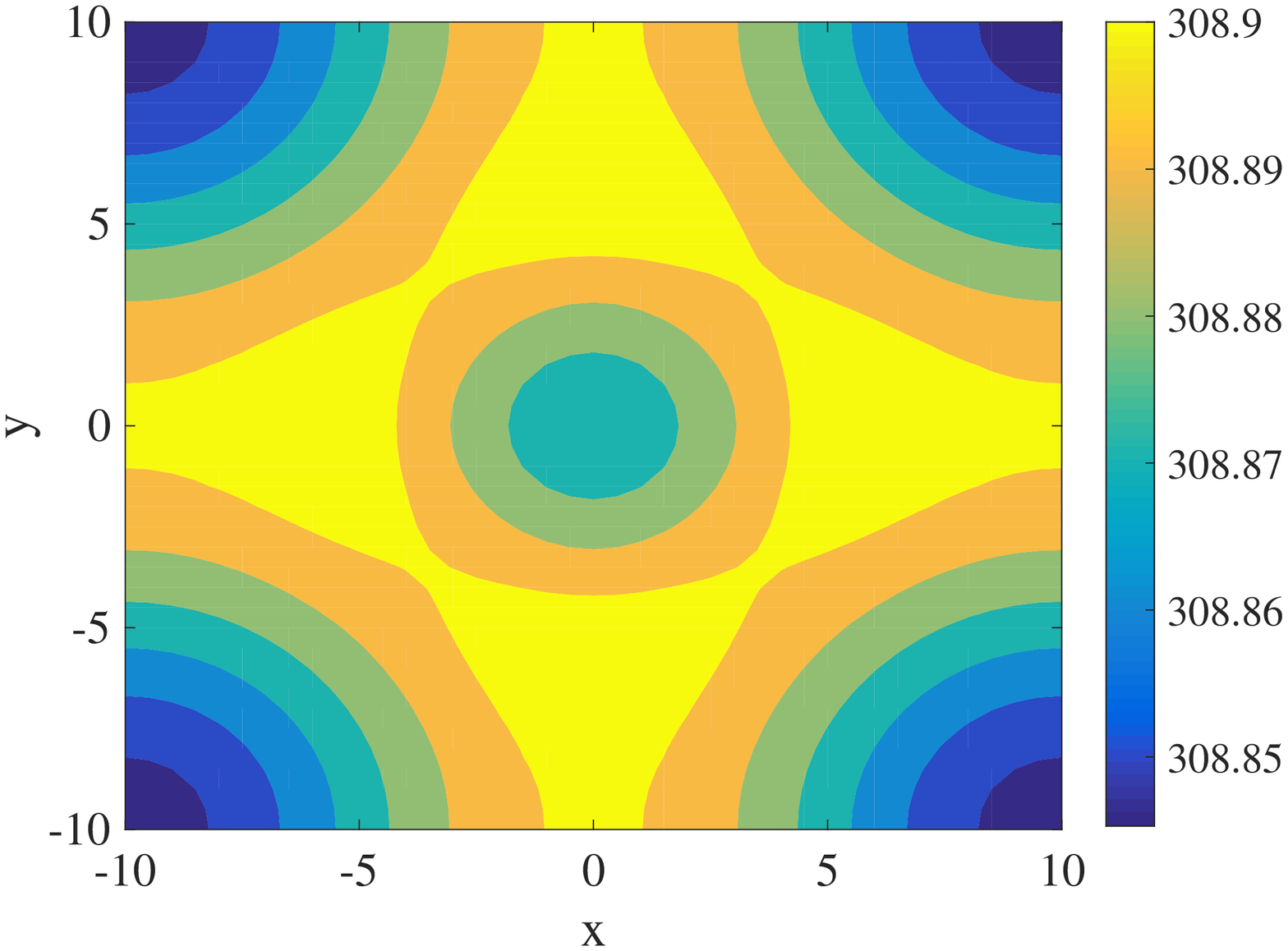}
            \end{minipage}
            }
           \centering \subfigure[$T$ at the 60th time step]{
            \begin{minipage}[b]{0.3\textwidth}
            \centering
             \includegraphics[width=0.95\textwidth,height=0.9\textwidth]{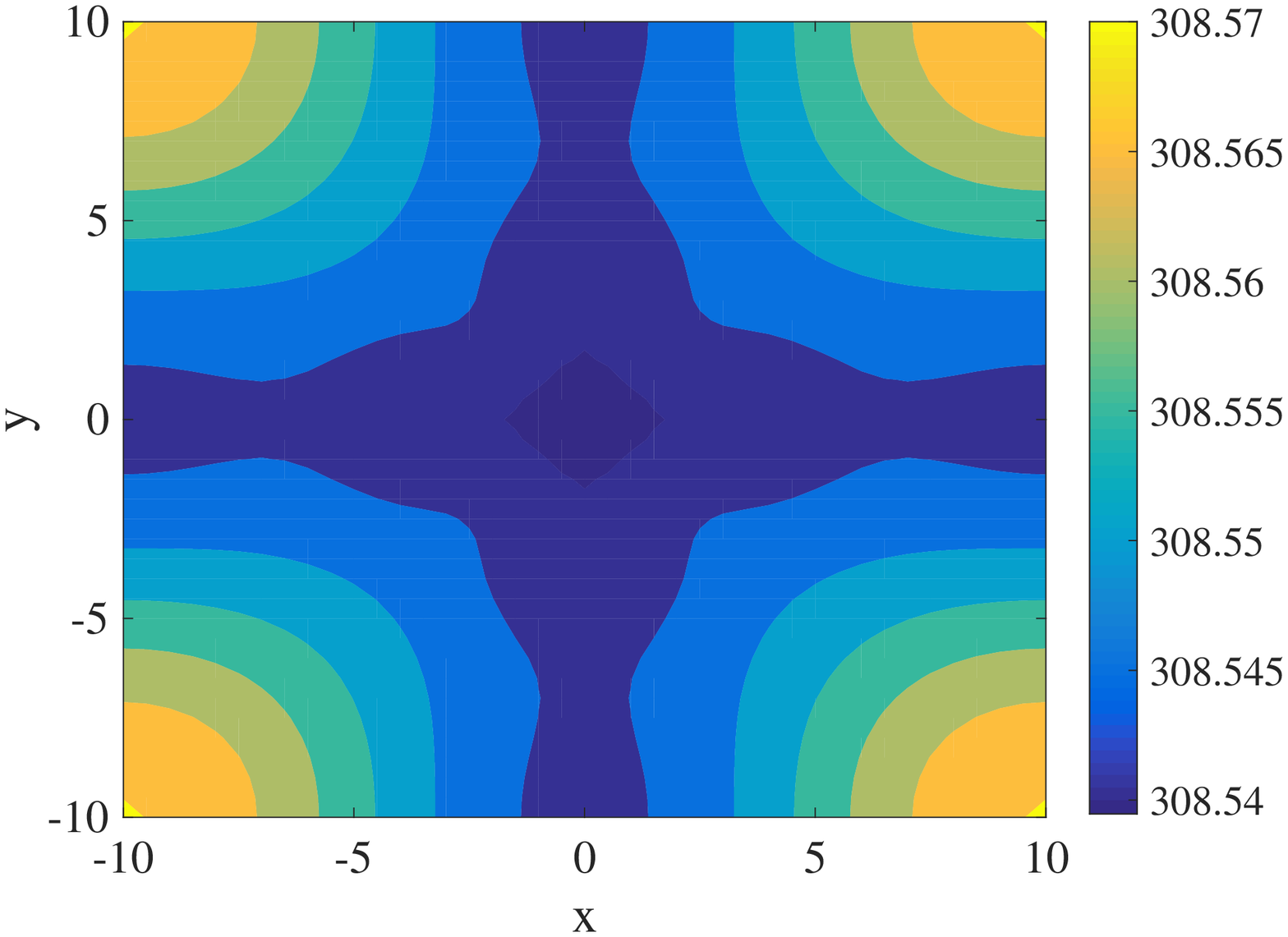}
            \end{minipage}
            }
         \caption{Example 1:  temperature ($T$) contours  at  different time steps.}
            \label{IsolatedSysC1andC5Temperature}
 \end{figure}

\begin{figure}
           \centering \subfigure[]{
            \begin{minipage}[b]{0.3\textwidth}
               \centering
             \includegraphics[width=\textwidth,height=0.9\textwidth]{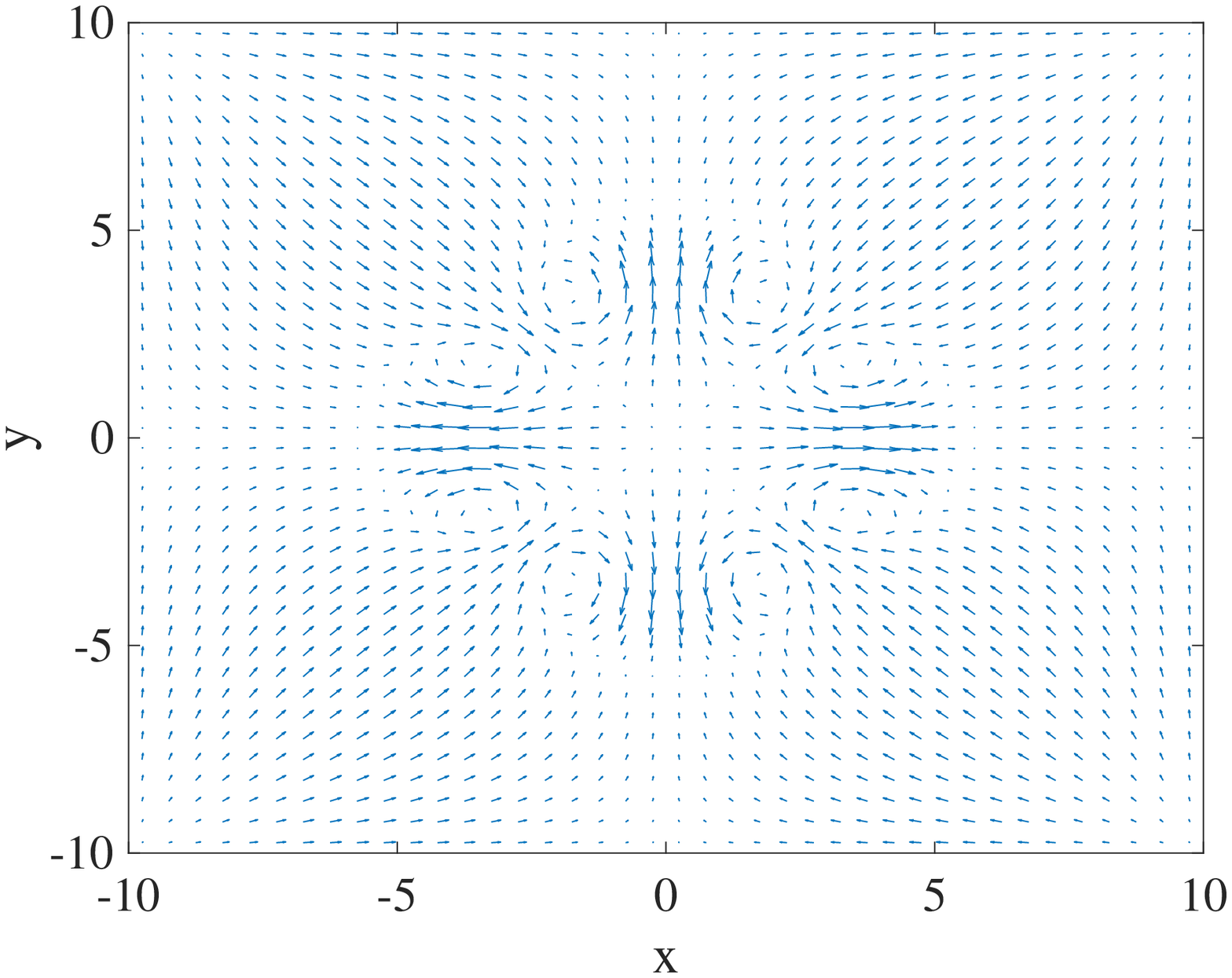}
            \end{minipage}
            }
            \centering \subfigure[]{
            \begin{minipage}[b]{0.3\textwidth}
            \centering
             \includegraphics[width=\textwidth,height=0.9\textwidth]{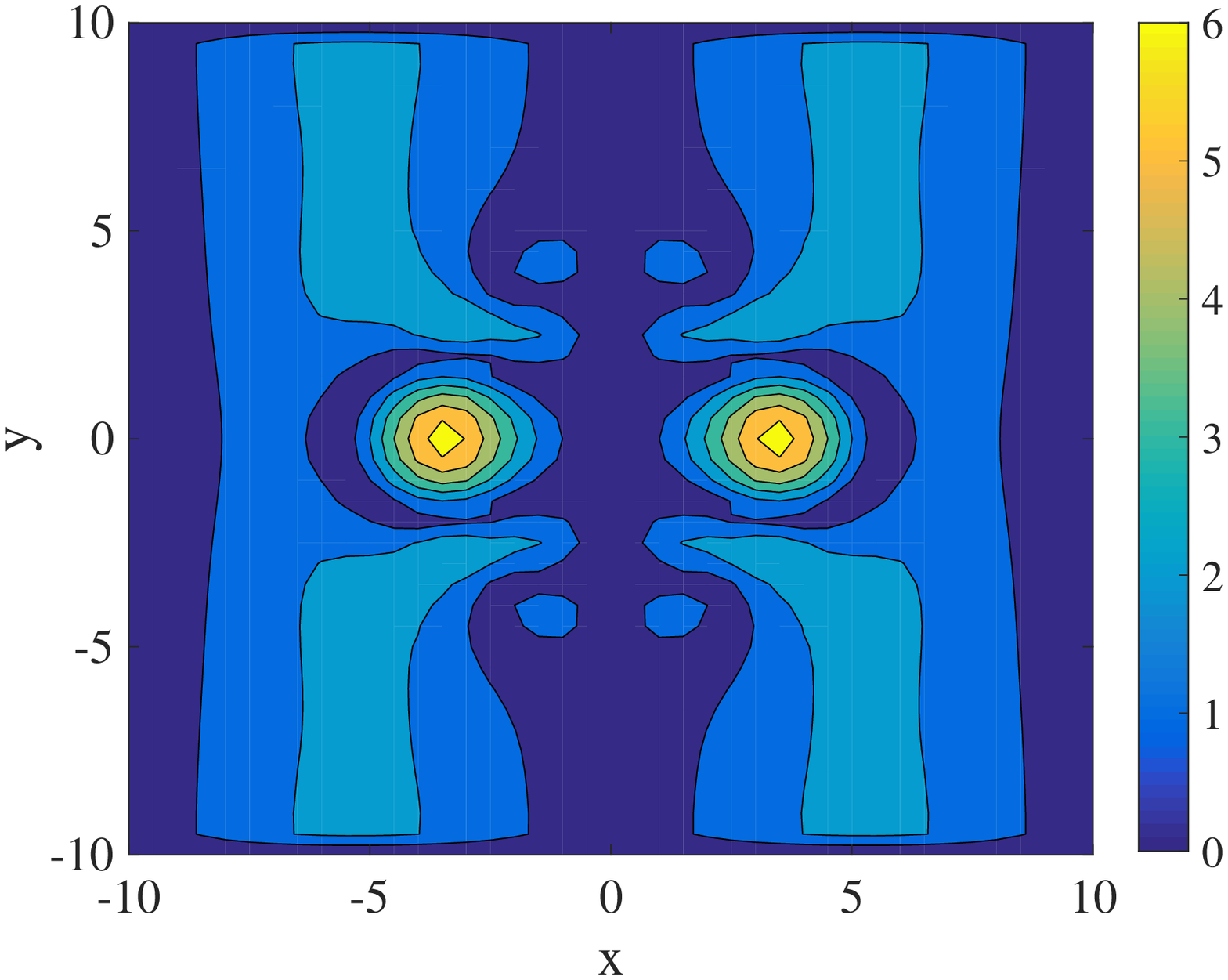}
            \end{minipage}
            }
           \centering \subfigure[]{
            \begin{minipage}[b]{0.3\textwidth}
               \centering
             \includegraphics[width=\textwidth,height=0.9\textwidth]{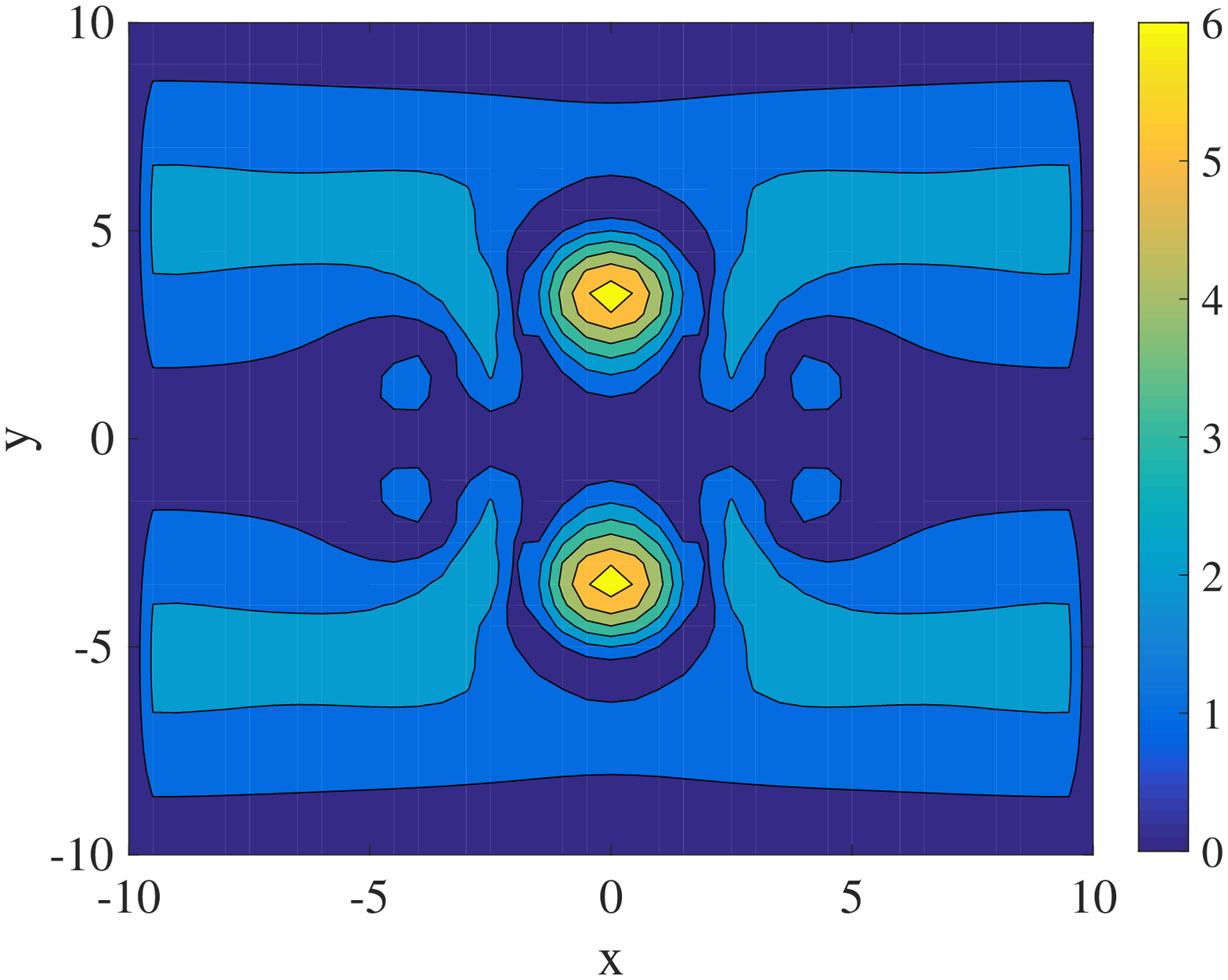}
            \end{minipage}
            }
            \centering \subfigure[]{
            \begin{minipage}[b]{0.3\textwidth}
               \centering
             \includegraphics[width=\textwidth,height=0.9\textwidth]{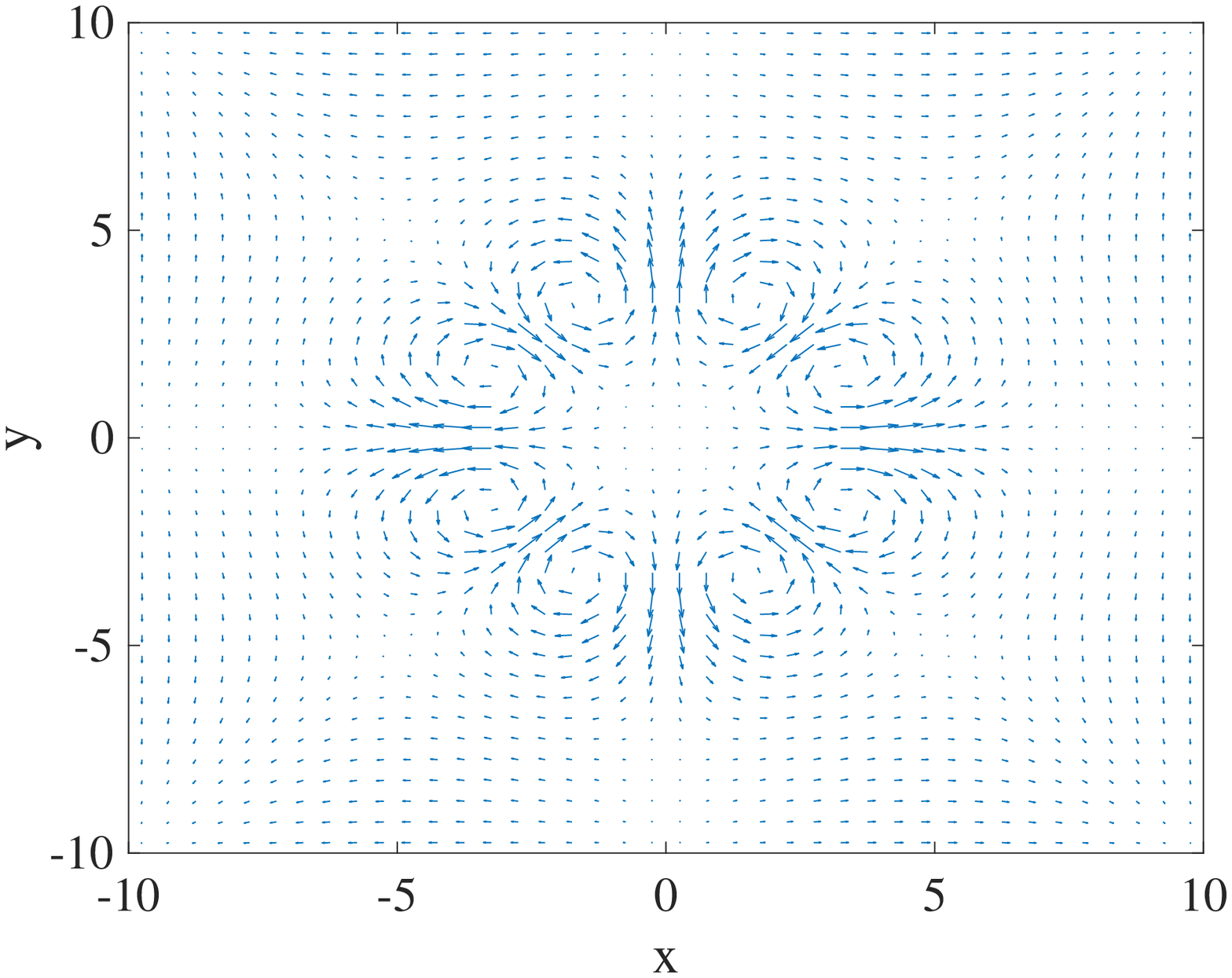}
            \end{minipage}
            }
            \centering \subfigure[]{
            \begin{minipage}[b]{0.3\textwidth}
            \centering
             \includegraphics[width=\textwidth,height=0.9\textwidth]{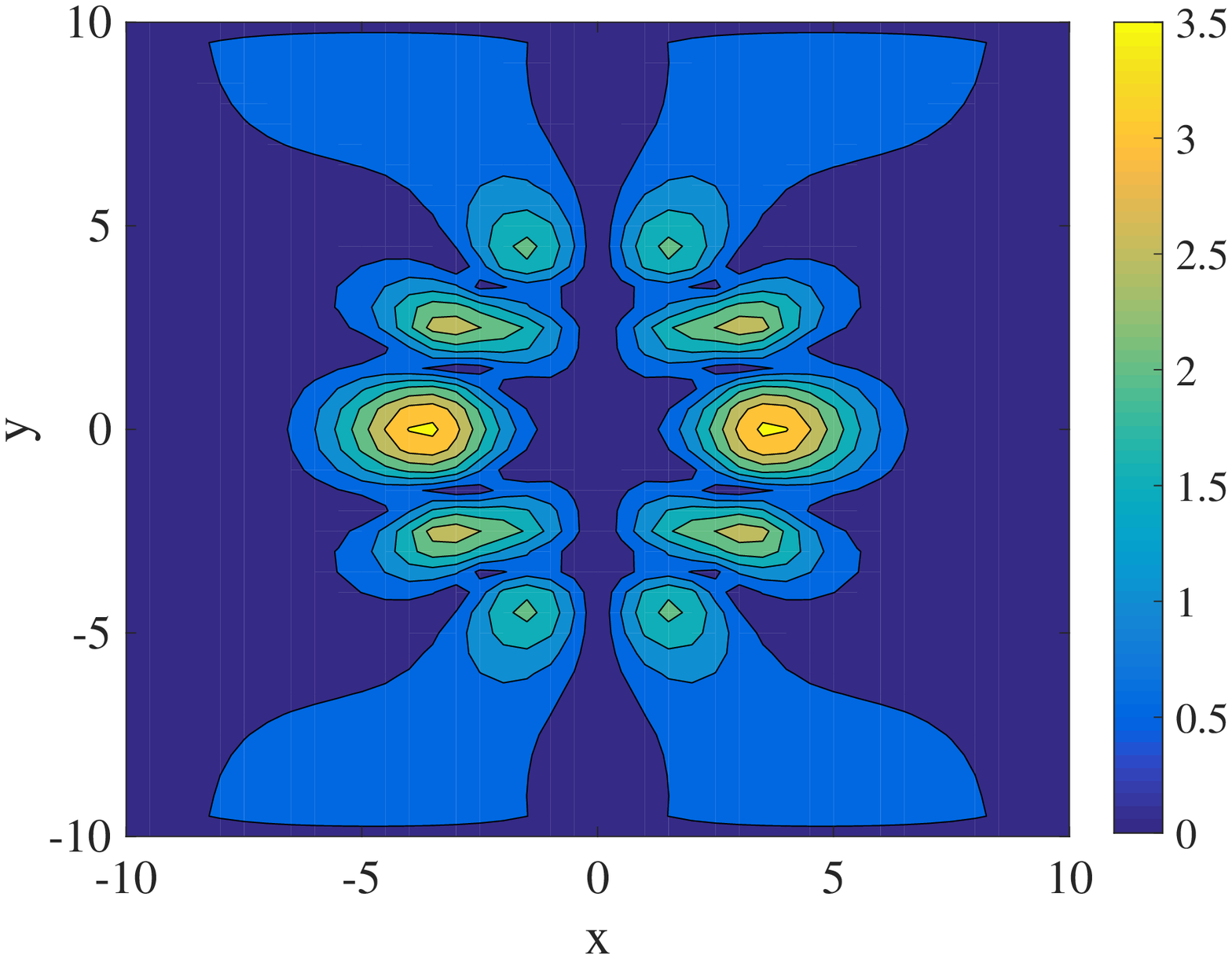}
            \end{minipage}
            }
           \centering \subfigure[]{
            \begin{minipage}[b]{0.3\textwidth}
               \centering
             \includegraphics[width=\textwidth,height=0.9\textwidth]{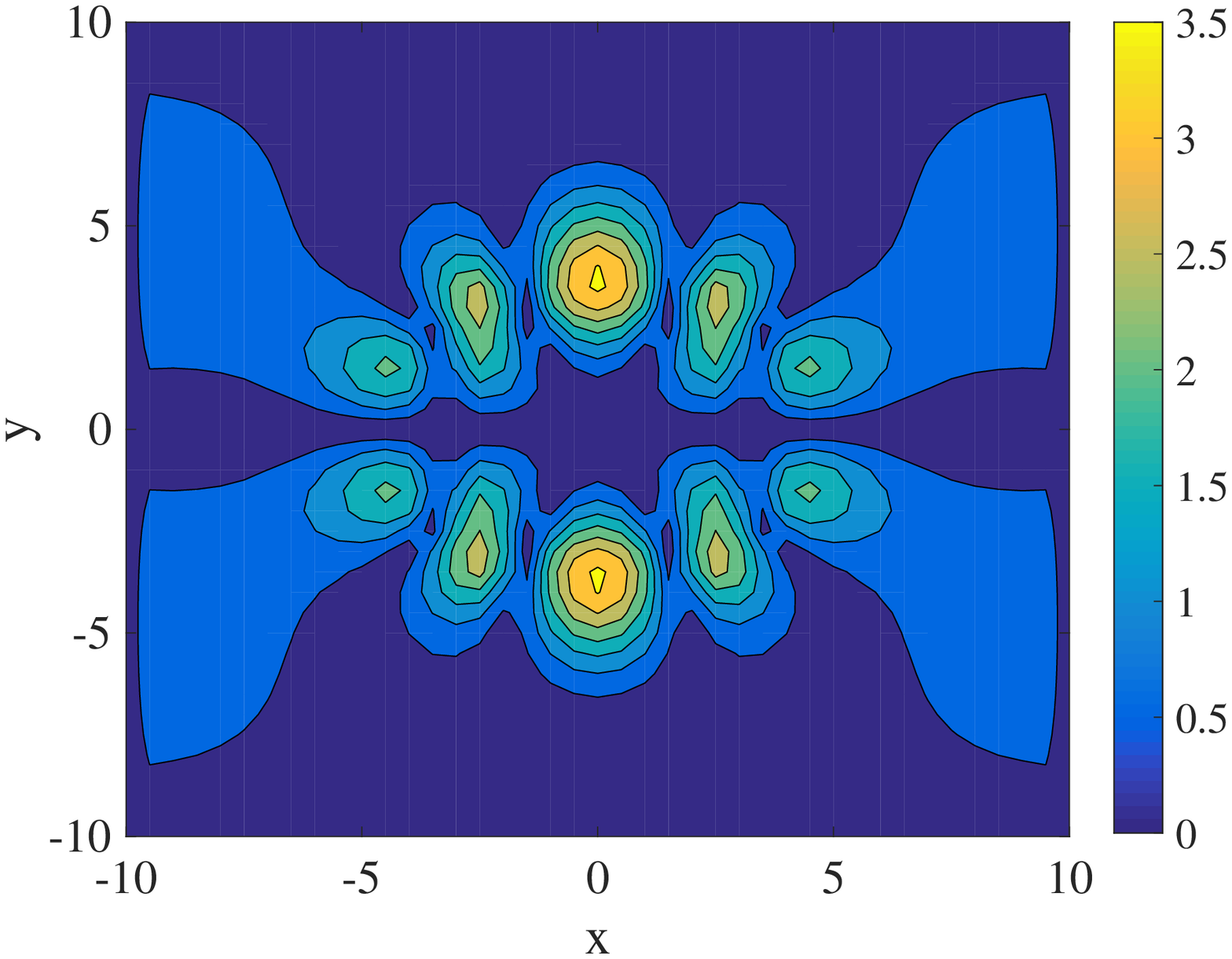}
             \end{minipage}
            }
          \caption{Example 1:   flow quivers (left column), magnitude contours of $x$-direction velocity component (center column), and magnitude contours of $y$-direction velocity component (right column)   at the 30th(top row) and 60th(bottom row) time step  respectively.}
            \label{IsolatedSysC1andC5Velocity}
 \end{figure}

\section{Conclusions}


A general model obeying the  thermodynamic laws  and Onsager's reciprocal relations and  characterizing compressibility and partial miscibility between multiple  fluids has been developed for   non-isothermal     multi-component diffuse-interface  two-phase  flows with realistic equations of state.  This model is formulised  by a set of nonlinear and coupling equations, including the component mass balance equations, the   momentum conservation equation and the total energy balance equation, which is unified for general reference velocities.  We have proved an important     relation between   pressure, chemical potential and temperature, and on the basis of it, we obtain   a new formulation of the  momentum conservation equation, which indicates   that the gradients of  chemical potentials and temperature become the primary driving force of the fluid motion except for the external forces.  
For numerical simulations, we have proposed an efficient, entropy stable numerical method based on  the convex-concave  splitting of Helmholtz free energy density with respect to molar densities and temperature.  We have proved unconditional  entropy stability of the proposed method by deriving the variations of Helmholtz free energy and kinetic energy with time steps.  Numerical results are provided to validate  the proposed  method.

\begin{appendix}\label{appendix}

\section{Helmholtz free energy density}
We describe the computations of Helmholtz free energy density $f_{b}(\n,T)$ of a homogeneous fluid determined by Peng-Robinson equation of state \cite{Peng1976EOS}. 

Let $R$ be the universal gas constant. 
We denote  by $T_{c_i}$ and $P_{c_i}$  the  $i$th component critical temperature and critical pressure, respectively.  For the $i$th component, let the reduced temperature be  $T_{r_i}=T/T_{c_i}$.  We let the mole fraction of component $i$ be $y_i=n_i/n$, where  $n=\sum_{i=1}^Mn_i$ is the overall molar density. 
The parameters $a_{i}$ and $b_{i}$ are calculated as
\begin{equation}\label{eqaibi}
   a_{i}= 0.45724\frac{R^2T_{c_i}^2}{P_{c_i}}\[1+m_i(1-\sqrt{T_{r_i}})\]^2,~~~~b_{i}= 0.07780\frac{RT_{c_i}}{P_{c_i}}.
\end{equation}
 The coefficients $m_i$ are calculated  by the following formulas
 \begin{subequations}\label{eqdefmi}
\begin{equation}
 m_i=0.37464 + 1.54226\omega_i-  0.26992\omega_i^2 ,~~\omega_i\leq0.49,
\end{equation}
\begin{equation}
 m_i=0.379642+1.485030\omega_i-0.164423\omega_i^2 +0.016666 \omega_i^3,~~\omega_i>0.49,
\end{equation}
\end{subequations}
where $\omega_i$ is the acentric factor.
 Finally,  $a(T)$ and $b$ are calculated by
\begin{equation}\label{eqab}
   a =\sum_{i=1}^M\sum_{j=1}^M y_i y_j \(a_ia_j\)^{1/2}(1-k_{ij}),~~~~b =\sum_{i=1}^M  y_i b_{i},
\end{equation}
where $k_{ij}$ the given binary interaction coefficients for the energy parameters.

The   correlation coefficients $\alpha_{ik}$  estimate the molar heat capacity of ideal gas at the constant pressure   as 
\begin{eqnarray}\label{eqHeatCapacity01}
    \psi_i^p(T)=\sum_{k=0}^3\alpha_{ik}T^k.
\end{eqnarray}
We list the   correlation coefficients of methane and pentane  in Table \ref{tabParametersheatcapacity}.

The bulk Helmholtz free energy density, denoted by $f_{b}$,  is calculated as a sum of three contributions
\begin{eqnarray}\label{eqHelmholtzEnergy_a0_01}
    f_b(n,T)&=& f_b^{\textnormal{ideal}}(n,T) + f_b^{\textnormal{repulsion}}(n,T)+f_b^{\textnormal{attraction}}(n,T),
\end{eqnarray}
where
\begin{eqnarray}\label{eqHelmholtzEnergy_a0_01}
    f_b^{\textnormal{ideal}}(n,T)&=& n\vartheta_0 -ns_0T+\sum_{i=1}^Mn_i\sum_{k=0}^3\alpha_{ik}\frac{T^{k+1}-T_0^{k+1}}{k+1}-nR(T-T_0)\nonumber\\
    &&-\sum_{i=1}^Mn_iRT\ln\(\frac{P_0}{n_iRT}\)-\sum_{i=1}^Mn_iT\int_{T_0}^T\frac{\psi_i^p(\xi)}{\xi}d\xi,
\end{eqnarray}
\begin{eqnarray}\label{eqHelmholtzEnergy_a0_02}
    f_b^{\textnormal{repulsion}}(n,T)=-nRT\ln\(1-bn\),
\end{eqnarray}
\begin{eqnarray}\label{eqHelmholtzEnergy_a0_03}
    f_b^{\textnormal{attraction}}(n,T)= \frac{a(T)n}{2\sqrt{2}b}\ln\(\frac{1+(1-\sqrt{2})b n}{1+(1+\sqrt{2})b n}\),
\end{eqnarray}
where  $T_0=  298.15 $K, $ P_0 = 1 $bar,  $\vartheta_0=-2478.95687512 $ J$/$mol and  $s_0=59.5827$ J$/($mol$\cdot$K) for the mixture of methane  and pentane in numerical tests.

\section{Internal energy and entropy}
The bulk internal energy, denoted by $\vartheta_b$,  is formulated as 
 \begin{eqnarray}\label{eqinternalenergy}
    \vartheta_b(n,T)&=&  n\vartheta_0 +\sum_{i=1}^Mn_i\sum_{k=0}^3\alpha_{ik}\frac{T^{k+1}-T_0^{k+1}}{k+1}-nR(T-T_0)\nonumber\\
    &&+\frac{n\(a(T)-Ta'(T)\)}{2\sqrt{2}b}\ln\(\frac{1+(1-\sqrt{2})b n}{1+(1+\sqrt{2})b n}\),
\end{eqnarray}
where $a'(T)$ denotes the  derivative with respect to $T$.
We  denote by $s_b$ the bulk entropy and  express it as 
 \begin{eqnarray}\label{eqinternalenergy}
    s_b(n,T)&=& ns_0+  nR\ln\(1-bn\)+\sum_{i=1}^Mn_iR\ln\(\frac{P_0}{n_iRT}\)+\sum_{i=1}^Mn_i\int_{T_0}^T\frac{\psi^p_i(\xi)}{\xi}d\xi \nonumber\\
    &&-\frac{na'(T)}{2\sqrt{2}b}\ln\(\frac{1+(1-\sqrt{2})b n}{1+(1+\sqrt{2})b n}\).
\end{eqnarray}
The above formulations are referred to  \cite{smejkal2017phase}.

\section{Some physical parameters}

The  influence parameters in numerical tests are taken as $$\bm{c}=(c_{ij})_{i,j=1}^2=\left(\begin{array}{cc}0.0282 & 0.0462 \\ 0.0462 & 0.3019\end{array}\right)\times1e-18.$$

We list some physical parameters of the substances in Tables \ref{tabParametersPREOS}, \ref{tabParametersheatcapacity} and \ref{tabbinaryinteractioncoefficients}. 
 \begin{table}[htp]
\caption{Physical parameters}
\begin{center}
\begin{tabular}{cccccc}
\hline
Substance & $P_c$(bar) & $T_c$(K) & Acentric factor & $M_w$(g/mole)\\
\hline
methane     &  45.99                          & 190.56                            &0.011              &16.04 \\
pentane      &  33.70                          & 469.7                              &0.251              &72.15\\
\hline
\end{tabular}
\end{center}
\label{tabParametersPREOS}
\end{table}

   \begin{table}[htp]
\caption{The   correlation coefficients in the heat capacity}
\begin{center}
\begin{tabular}{cccccc}
\hline
Substance & $\alpha_0$ & $\alpha_1$ & $\alpha_2$ & $\alpha_3$\\
\hline
methane     &  19.25        & 5.213e-2  & 1.197e-5         & -1.132e-8 \\
pentane      &  -3.626       &  4.873e-1 & -2.580e-4        & 5.305e-8\\
\hline
\end{tabular}
\end{center}
\label{tabParametersheatcapacity}
\end{table}

\begin{table}[htp]
\caption{Binary interaction coefficients}
\begin{center}
\begin{tabular}{cccccc}
\hline
                   & methane    & pentane    \\
\hline
methane     &  0                                 & 0.041               \\
pentane      &  0.041                          & 0                     \\
\hline
\end{tabular}
\end{center}
\label{tabbinaryinteractioncoefficients}
\end{table}

\end{appendix}

\small

\begin{thebibliography}{10}


\bibitem{Abels2012TwoPhaseModel}
H. Abels, H. Garcke,   G. Gr{\"u}n.
\newblock Thermodynamically consistent, frame indifferent diffuse interface models for incompressible two-phase flows with different densities. 
\newblock {\em Mathematical Models and Methods in Applied Sciences},   Vol. 22, No. 3, 1150013, 2012.

\bibitem{Alpak2016PhaseField}
F. O.  Alpak, B. Riviere,  F. Frank.
\newblock A phase-field method for the direct simulation of two-phase flows in pore-scale media using a non-equilibrium wetting boundary condition.
\newblock {\em Computational Geosciences},  20: 881--908, 2016.  

\bibitem{Anderson1998DiffuseInterfaceMethods}
D. M. Anderson,  G. B. McFadden,  A. A. Wheeler.
\newblock Diffuse-Interface Methods in Fluid Mechanics.
\newblock {\em Annual Review of Fluid Mechanics},   30: 139-165, 1998.

\bibitem{arbogast1997mixed}
T.~Arbogast, M.F. Wheeler,  I.~Yotov.
\newblock Mixed finite elements for elliptic problems with tensor coefficients
  as cell-centered finite differences.
\newblock {\em SIAM Journal on Numerical Analysis}, 4(2):  828--852, 1997.


\bibitem{Bao2012FEM}
K. Bao, Y. Shi, S. Sun,  X.-P. Wang.
\newblock A finite element method for the numerical solution of the coupled Cahn-Hilliard and Navier-Stokes system for moving contact line problems.
\newblock {\em Journal of Computational Physics}, 231(24): 8083--8099,  2012.


\bibitem{Baskaran2013convexsplitting}
A. Baskaran, J. Lowengrub, C. Wang, S. Wise.
\newblock Convergence analysis of a second order convex splitting scheme for the modified phase field crystal equation.
\newblock {\em SIAM Journal on Numerical Analysis}, 51(5): 2851--2873,  2013.

\bibitem{Boyer2006MCH}
F. Boyer,  C. Lapuerta.
\newblock Study of a three component Cahn-Hilliard flow model.
\newblock {\em ESAIM:  Mathematical Modelling and Numerical Analysis}, 40(4): 653--687, 2006.

\bibitem{Bueno2016liquid} 
J. Bueno, H. Gomez.
\newblock {Liquid-vapor transformations with surfactants. Phase-field model and Isogeometric Analysis}.
\newblock {\em Journal of Computational Physics}, 321: 797--818, 2016.



\bibitem{CahnHilliard1958}
 J.~W.  Cahn, J.~E. Hilliard.
\newblock {Free Energy of a Nonuniform System. I. Interfacial Free Energy}.
\newblock {\em Journal of Chemical Physics}, 28: 258-267, 1958.

\bibitem{Chaudhri2014VDWApp}
A. Chaudhri, J. B. Bell, A. L. Garcia,  A. Donev.
\newblock Modeling multiphase flow using fluctuating hydrodynamics.
\newblock {\em Physical Review E},   90, 033014, 2014.


\bibitem{chen2006multiphase} Z. Chen, G.  Huan,  Y.  Ma,   Computational methods for multiphase flows
in porous media. SIAM Comp. Sci. Eng., Philadelphia, 2006.

\bibitem{shen2016JCP} 
Y. Chen, J.  Shen. 
\newblock Efficient, adaptive energy stable schemes for the incompressible Cahn-Hilliard Navier-Stokes phase-field models.
\newblock {\em  Journal of Computational Physics, }  308:  40-56, 2016.


\bibitem{Groot2015NET}
S. R. De Groot,   P. Mazur.
\newblock{\em Non-Equilibrium Thermodynamics.}
\newblock  Dover Publications, New York,   2011.

\bibitem{Ding2007PhaseField}
H. Ding, P.D.M. Spelt, C.  Shu.
\newblock Diffuse interface model for incompressible two-phase flows with large density ratios.
\newblock {\em Journal of Computational Physics},  226:  2078--2095, 2007.

\bibitem{Dong2017PhaseField}
S. Dong.
\newblock{ Wall-bounded multiphase flows of N immiscible incompressible fluids: Consistency and contact-angle boundary condition.}
\newblock {\em Journal of Computational Physics},  338: 21--67,  2017.



\bibitem{Elliott1993ConvexSpliting} 
C. M. Elliott,  A. M. Stuart, 
\newblock The global dynamics of discrete semilinear parabolic equations,
\newblock {\em SIAM Journal on Numerical Analysis,} 30: 1622--1663, 1993.

\bibitem{Emmerich2011PhaseField}
H. Emmerich.
\newblock{\em The Diffuse Interface Approach in Materials Science:
Thermodynamic Concepts and Applications
of Phase-Field Models.}
\newblock  Springer,   2011.


\bibitem{Eyre1998ConvexSplitting}
D.~J. Eyre. 
\newblock Unconditionally gradient stable time marching the Cahn-Hilliard equation. 
 \newblock {\em Computational and mathematical models of microstructural evolution (San Francisco, CA, 1998),}
  Mater. Res. Soc. Sympos. Proc., 529: 39--46. MRS, Warrendale, PA, 1998.

\bibitem{fan2017componentwise}
 X. Fan,  J. Kou,  Z. Qiao,   S. Sun.
\newblock A Componentwise Convex Splitting Scheme for Diffuse Interface Models with Van der Waals and Peng--Robinson Equations of State.
 \newblock {\em SIAM Journal on Scientific Computing,} 39(1): B1--B28, 2017.

\bibitem{firoozabadi1999thermodynamics}
A. Firoozabadi.
\newblock {\em Thermodynamics of hydrocarbon reservoirs.}
\newblock McGraw-Hill New York, 1999.

\bibitem{Girault1996Mac}
 V. Girault,  H. Lopez.
\newblock Finite-element error estimates for the MAC scheme.
 \newblock {\em IMA Journal of Numerical Analysis,} 16(3): 347-379, 1996.

\bibitem{Gonnella2008binaryVDW}
G. Gonnella, A. Lamura,  A. Piscitelli.
\newblock Dynamics of binary mixtures in inhomogeneous temperatures.
 \newblock {\em Journal of Physics A: Mathematical and Theoretical}, 41, 105001, 2008.
 
 \bibitem{Guo2015phasefield}
 Z. Guo, P. Lin
 \newblock A thermodynamically consistent phase-field model for two-phase flows with thermocapillary effects.
 \newblock {\em Journal of Fluid Mechanics},    766: 226--271, 2015.
 
 \bibitem{Guo2016convexsplitting}
 J. Guo, C. Wang, S. Wise,  X. Yue.
 \newblock An $H^2$ convergence of a second-order convex-splitting, finite difference scheme for the three-dimensional Cahn-Hilliard equation.
\newblock {\em  Communications in Mathematical Sciences}, 14: 489-515, 2016.

\bibitem{Hoteit2013diffusion}
H. Hoteit.
\newblock  Modeling diffusion and gas-oil mass transfer in fractured reservoirs. 
\newblock {\em Journal of Petroleum Science and Engineering},  105:~1--17,  2013.



\bibitem{jindrova2013fast}
T.~Jindrov{\'a},  J.~Miky$\check{\textnormal{s}}$ka.
\newblock Fast and robust algorithm for calculation of two-phase equilibria at
  given volume, temperature, and moles.
\newblock {\em Fluid Phase Equilibria}, 353:~101--114, 2013.
 

\bibitem{mikyvska2015General}
 T.~Jindrov{\'a},  J.~Miky$\check{\textnormal{s}}$ka.
\newblock General algorithm for multiphase equilibria calculation at given volume, temperature, and moles.
\newblock {\em Fluid Phase Equilibria}, 393:~7--25, 2015.


\bibitem{Kim2009PhaseField}
J. Kim. 
\newblock A generalized continuous surface tension force formulation for phase-field models for multi-component immiscible fluid flows.
\newblock {\em Computer Methods in Applied Mechanics
and Engineering}, 198: 3105--3112, 2009.

\bibitem{Kim2012PhaseField}
J. Kim. 
\newblock Phase-field models for multi-component fluid flows.
\newblock {\em Communications in Computational Physics},  2(3):  613--661,  2012.


 \bibitem{kousun2015CMA} 
J. Kou, S. Sun,  X. Wang.
\newblock Efficient numerical methods for simulating surface tension of multi-component mixtures  with the gradient theory of  fluid interfaces. \newblock {\em Computer Methods in Applied Mechanics
and Engineering}, 292: 92--106, 2015.

\bibitem{kousun2015SISC}
J. Kou, S. Sun.
\newblock Numerical methods for a multi-component two-phase   interface model with  geometric mean influence parameters.
\newblock {\em SIAM Journal on Scientific Computing}, 37(4): B543--B569, 2015.

\bibitem{kousun2015CHE}
J. Kou, S. Sun.
\newblock Unconditionally stable methods for simulating multi-component two-phase interface models with Peng-Robinson equation of state and various boundary conditions.
\newblock {\em Journal of Computational and Applied Mathematics}, 291(1): 158--182, 2016.  

\bibitem{kousun2016Flash}
J. Kou, S. Sun,  X. Wang.
\newblock An energy stable evolution method for simulating two-phase equilibria of multi-component fluids at constant moles, volume and temperature.
\newblock {\em Computational Geosciences},  20: 283--295, 2016.  


\bibitem{kouandsun2016multiscale}
J.~Kou,  S.~Sun.
\newblock {Multi-scale     diffuse interface modeling  of multi-component two-phase  flow  with partial miscibility.}
\newblock {\em Journal of Computational Physics}, 318: 349--372, 2016.



\bibitem{kouandsun2017modeling}
J.~Kou, S.~Sun.
\newblock {Thermodynamically consistent modeling and  simulation of  multi-component two-phase  flow   with partial miscibility}.
\newblock {\em Computer Methods in Applied Mechanics and Engineering},     331:  623--649, 2018.




\bibitem{kou2018nvtpc} 
J.~Kou,  S.~Sun.
\newblock {A stable algorithm for calculating phase equilibria with capillarity at specified moles, volume and temperature using a dynamic model}.
\newblock {\em Fluid Phase Equilibria},
  456: 7--24, 2018.

\bibitem{kou2017nonisothermal} 
J.~Kou,  S.~Sun.
\newblock {Thermodynamically consistent   simulation of  nonisothermal  diffuse-interface two-phase flow  with  Peng-Robinson equation of state},
	arXiv:1712.03090 [math.NA], 2017.

\bibitem{li2001prediction}
Z. Li,  B. C.-Y. Lu.
\newblock On the prediction of surface tension for multicomponent mixtures.
\newblock {\em The Canadian Journal of Chemical Engineering}, 79(3):~402--411,
  2001.

    
  \bibitem{Li2017IEQPREOS} 
 H. Li,  L. Ju, C. Zhang, Q. Peng.
\newblock Unconditionally energy stable linear schemes for the diffuse interface model with Peng-Robinson equation of state.
\newblock {\em Journal of Scientific Computing}, DOI 10.1007/s10915-017-0576-7, 2017.

 \bibitem{Li2016MCH} 
Y. Li, J.-I. Choi,   J. Kim.
\newblock Multi-component Cahn-Hilliard system with different boundary conditions in complex domains.
\newblock {\em Journal of Computational Physics}, 323: 1--16, 2016.

 \bibitem{Li2017convex} 
Y. Li, J. Kou, S. Sun.
\newblock Numerical modeling of isothermal compositional grading by convex splitting methods.
\newblock {\em Journal of Natural Gas Science and Engineering},  43: 207--221, 2017.

\bibitem{liu2015liquid}
J. Liu, C. M. Landis, H. Gomez, T. J.R. Hughes.
\newblock {Liquid--vapor phase transition: Thermomechanical theory, entropy stable numerical formulation, and boiling simulations}.
\newblock {\em Computer Methods in Applied Mechanics and Engineering}, 297: 476--553, 2015.

\bibitem{liu2016binaryPRE}
J. Liu, G. Amberg,  M. Do-Quang. 
\newblock{Diffuse interface method for a compressible binary fluid}.
\newblock {\em Physical Review E},   93, 013121, 2016.
  

  \bibitem{Firoozabadi2007Diffusion}
A. Leahy-Dios,  A. Firoozabadi.
\newblock Unified Model for Nonideal Multicomponent Molecular Diffusion Coefficients.
\newblock {\em AIChE Journal}, 53(11): 2932--2939, 2007.


\bibitem{mikyvska2011new} 
 J. Miky$\check{\textnormal{s}}$ka,  A. Firoozabadi.
\newblock A new thermodynamic function for phase-splitting at constant
  temperature, moles, and volume.
\newblock {\em AIChE Journal}, 57(7):1897--1904, 2011.

\bibitem{mikyvska2012investigation}
 J.~Miky$\check{\textnormal{s}}$ka,  A.~Firoozabadi.
\newblock Investigation of mixture stability at given volume, temperature, and
  number of moles.
\newblock {\em Fluid Phase Equilibria}, 321:1--9, 2012.


\bibitem{miqueu2004modelling}
C. Miqueu, B. Mendiboure, C. Graciaa,  J. Lachaise.
\newblock Modelling of the surface tension of binary and ternary mixtures with the gradient theory of fluid interfaces.
\newblock {\em Fluid Phase Equilibria}, 218:~189--203, 2004.



\bibitem{Moortgat2013compositional}
J. Moortgat,  A. Firoozabadi.
\newblock Higher-order compositional modeling of three-phase flow in 3D fractured porous media based on cross-flow equilibrium.
\newblock {\em Journal of Computational Physics}, 250: 425--445,  2013.


\bibitem{Nagarajan1991}
N.R.~Nagarajan, A.S.~Cullick.
\newblock New strategy for phase equilibrium and critical point calculations by thermodynamic energy analysis. Part I. Stability analysis and flash.
\newblock {\em Fluid Phase Equilibria},   62(3):  191--210, 1991.


  
 \bibitem{Onuki2005PRL}
A. Onuki.  
\newblock Dynamic van der Waals theory of two-phase fluids in heat flow. 
\newblock {\em Physical Review Letters,}   94(5): 054501, 2005.

 \bibitem{Onuki2007PRE}
A. Onuki.  
\newblock  Dynamic van der Waals theory. 
\newblock {\em Physical Review E,}   75(3): 036304, 2007.



\bibitem{Pecenko2010approachVanderWaals}
A. Pecenko, J.G.M. Kuerten, C.W.M. van der Geld.
\newblock A diffuse-interface approach to two-phase isothermal flow of a Van der Waals fluid near the critical point,
\newblock {\em International Journal of Multiphase Flow},  36: 558--569, 2010.

\bibitem{Pecenko2011VanderWaalsModel}
A. Pecenko, L.G.M. van Deurzen, J.G.M. Kuerten, C.W.M. van der Geld.
\newblock Non-isothermal two-phase flow with a diffuse-interface model.
\newblock {\em International Journal of Multiphase Flow}, 37: 149-165, 2011.

\bibitem{Peng1976EOS}
D. Peng,  D.B. Robinson.
\newblock A new two-constant equation of state.
\newblock {\em Industrial and Engineering Chemistry Fundamentals},
  15(1): 59--64, 1976.

\bibitem{Peng2017convexsplitting}
Q. Peng.
\newblock A convex-splitting scheme for a diffuse interface model with Peng-Robinson equation of state.
\newblock {\em Advances in Applied Mathematics and Mechanics}, 9(5): 1162--1188, 2017.

\bibitem{polivka2014compositional}
O.~Pol{\'\i}vka,  J.~Miky$\check{\textnormal{s}}$ka.
\newblock Compositional modeling in porous media using constant volume flash
  and flux computation without the need for phase identification.
\newblock {\em Journal of Computational Physics}, 272:149--169, 2014.
  
\bibitem{qiaosun2014}
Z. Qiao,  S. Sun.
\newblock Two-phase fluid simulation using a diffuse interface model with Peng-Robinson equation of state.
\newblock {\em SIAM Journal on Scientific Computing}, 36(4): B708--B728, 2014.

\bibitem{Rasheed2013phasefield}
A. Rasheed,  A. Belmiloudi.
\newblock Mathematical modelling and numerical simulation of dendrite growth using phase-field method with a magnetic field effect.
\newblock {\em Communications in Computational Physics}, 14: 477--508, 2013.

\bibitem{shen2015SIAM}
 J. Shen, X. Yang.  
 \newblock Decoupled, energy stable schemes for phase-field models of two-phase incompressible flows.
 \newblock {\em  SIAM Journal on Numerical Analysis,}  53(1): 279-296, 2015.

\bibitem{smejkal2017phase}
 T. Smejkal, J. Miky$\check{\textnormal{s}}$ka.  
 \newblock Phase stability testing and phase equilibrium calculation at specified internal energy, volume, and moles.
 \newblock {\em  Fluid Phase Equilibria,}  431:  82--96, 2017.
  
\bibitem{Qian2016heatflow}
 M. T. Taylor,   T. Qian.
\newblock Thermal singularity and contact line motion in pool boiling: Effects of substrate wettability. 
\newblock {\em Physical Review E, }  93(3), 033105, 2016.

  
\bibitem{Tryggvason2011book}
G. Tryggvason, R. Scardovelli, S. Zaleski.
\newblock {\em Direct Numerical Simulations of Gas-Liquid Multiphase Flows.}
\newblock { Cambridge University Press, New York}, 2011.
  
  
\bibitem{Wise2009Convex}  
 S.~M. Wise, C. Wang, J.~S.  Lowengrub. 
 \newblock An energy-stable and convergent finite-difference scheme for the phase field crystal equation. 
 \newblock {\em  SIAM Journal on Numerical Analysis,} 47(3): 2269--2288, 2009.
 
 \bibitem{Yang2017MCH} 
 X. Yang, J. Zhao, Q. Wang,   J. Shen.
 \newblock Numerical approximations for a three-component Cahn-Hilliard phase-field model based on the invariant energy quadratization method.
 \newblock {\em Mathematical Models and Methods in Applied Sciences},   27(11): 1993--2030, 2017.
 
 
 \bibitem{Zhang2016Phasefield} 
Q. Zhang, X.-P. Wang.
\newblock Phase field modeling and simulation of three-phase flow on solid surfaces.
\newblock {\em Journal of Computational Physics}, 319: 79--107, 2016.



\end{thebibliography}

\end{document}